\documentclass[a4paper,english,oneside]{article}
\usepackage[T1]{fontenc}
\usepackage[utf8]{inputenc} 

\usepackage{geometry}
\usepackage{amsmath}
\usepackage{amsfonts}
\usepackage{amssymb}
\usepackage{amsthm}
\usepackage[english]{babel}
\usepackage{color}
\usepackage{hyperref}
\hypersetup{hidelinks}

\usepackage[cal=boondoxo,bb=ams]{mathalfa}

\theoremstyle{plain}
\newtheorem{theorem}{Theorem}[section]
\newtheorem{lemma}[theorem]{Lemma}
\newtheorem{corollary}[theorem]{Corollary}
\newtheorem{proposition}[theorem]{Proposition}

\theoremstyle{definition}
\newtheorem{definition}[theorem]{Definition}

\theoremstyle{remark}
\newtheorem{remark}{Remark}

    \DeclareMathOperator\supp{supp}

 \date{}
     
\begin{document}

\title{Lifespan estimates for local solutions to the semilinear wave equation in Einstein~--~de~Sitter spacetime}

\author{Alessandro Palmieri} 
\maketitle

\begin{abstract}
In this paper, we prove some blow-up results for the semilinear wave equation in generalized Einstein-de Sitter spacetime by using an iteration argument and we derive upper bound estimates for the lifespan. In particular, we will focus on the critical cases which require the employment of a slicing procedure in the iterative mechanism. Furthermore, in order to deal with the main critical case, we will introduce a non-autonomous and parameter dependent Cauchy problem for a linear ODE of second order, whose explicit solution will be determined by applying the theory of special functions.
\end{abstract}

\begin{flushleft}
\end{flushleft}

\begin{flushleft}
\textbf{AMS Classification (2010)} Primary: 35B44, 35L05, 35L71 ; Secondary: 33C10
\end{flushleft}


\section{Introduction}

In the last four decades, the proof of the Strauss conjecture concerning the critical exponent of the initial value problem for the semilinear wave equation with power nonlinearity required the effort of many mathematicians worldwide. Nowadays, we know that the critical exponent for the Cauchy problem
\begin{align*}
\begin{cases}  v_{tt} - \Delta v=|v|^p & x\in \mathbb{R}^n, \ t>0, \\
v(0,x)= \varepsilon v_0(x) & x\in \mathbb{R}^n, \\
 v_t(0,x)= \varepsilon v_1(x) & x\in \mathbb{R}^n,
\end{cases}
\end{align*} is the so -- called \emph{Strauss exponent} $p_{\mathrm{Str}}(n)$  (cf.  \cite{Joh79,Kato80,Gla81-g,Gla81-b,Sid84,Sch85,Zhou95,LS96,GLS97,YZ06,Zhou07}), that is, the positive root of the quadratic equation
\begin{align*}
(n-1)p^2-(n+1)p-2=0.
\end{align*} 
We are also interested in not only the critical exponent but also lifespan, the maximal existence time of the solution,
when the global in time existence cannot be expected.
See the introduction of \cite{IKTW19} for the complete picture of the lifespan estimates for the classical semilinear wave equation with power nonlinearity.

While the situation is completely understood in the Euclidean case with flat metric on $\mathbb{R}^n$, in the last years several papers have been devoted to study the semilinear wave equation in the spacetime $\mathbb{R}^{1+n}_+$ equipped with different Lorentzian metrics. The semilinear wave equation in Schwarzschild has been investigated in \cite{CG06,MMTT10,LMSTW14,LLM20} in the $1+3$ dimensional case. Moreover, the wave (or Klein-Gordon) equation in de Sitter and anti -- de Sitter spacetimes  have been investigated in the linear and semilinear case in \cite{Yag09,YagGal09,Yag12,GalYag2017,ER18,Gal18} and \cite{Gal03,YagGal08,YG09,YG09Rend}, respectively. Finally, the wave equation in Einstein --  de Sitter spacetime has been considered in \cite{GalKinYag10,GalYag14,GalYag17EdS}. In this paper, we shall examine the semilinear wave equation with power nonlinearity in a \emph{generalized Einstein -- de Sitter spacetime}. More precisely, let us consider the semilinear equation with singular coefficients
\begin{align} \label{Semi EdeS k Psi}
\varphi_{tt} -t^{-2k} \Delta \varphi +2t^{-1} \varphi_t =|\varphi|^p,
\end{align} where $k\in [0,1)$ and $p>1$. We call this model the semilinear wave equation in a generalized EdeS spacetime since for $k=2/3$ and $n=3$ Equation \eqref{Semi EdeS k Psi} is the semilinear wave equation in Einstein -- de Sitter (EdeS) spacetime with power nonlinearity.

In \cite[Theorem 1.3]{GalYag17EdS} the authors proved that for $$1<p<\max\big\{p_0(n,k),p_1(n,k)\big\}$$ a local in time solution to the corresponding Cauchy problem (with initial data prescribed at the initial time $t=1$) blows up in finite time, provided that the initial data fulfill certain integral sign conditions. Here $p_0(n,k)$ is the positive root of the quadratic equation
\begin{align}\label{intro equation critical exponent general case}
\left((1-k)n +1\right)p^2- \left((1-k)n +3+2k\right)p -2(1-k)=0,
\end{align} while 
\begin{align} \label{intro def p1}
p_1(n,k) \doteq 1+ \frac{2}{(1-k)n}.
\end{align} 
Furthermore, in \cite{GalYag17EdS} it is also shown that, for the semilinear wave equation in EdeS spacetime, the blow -- up is  the effect of the semilinear term. For this reason we shall focus our analysis on the effect of the nonlinear term, prescribing the Cauchy data at the initial time $t=1$.

Performing the transformation $u=t \varphi$, \eqref{Semi EdeS k Psi} becomes equivalent to the following semilinear equation for $u$
\begin{align} \label{Semi EdeS k u}
u_{tt} -t^{-2k} \Delta u  =t^{1-p}|u|^p.
\end{align}

In this paper, we investigate the blow -- up dynamic for \eqref{Semi EdeS k u} and, in particular, we will focus on the upper bound estimates for the lifespan and on the treatment of the critical case $p=\max\{p_0(n,k),p_1(n,k)\}$. More precisely, in the next sections we are going to provide a complete picture of the upper bound estimates for the lifespan of local in time solutions to \eqref{Semi EdeS k u} when $1<p\leqslant \max\{p_0(n,k),p_1(n,k)\}$.

In the subcritical case, we employ a Kato -- type lemma on the blow -- up dynamic for a second order ordinary differential inequality. On the other hand, in the critical case an iteration argument combined with a slicing procedure is applied. More particularly, for $p=p_0(n,k)$ we adapt the approach from \cite{WakYor18,WakYor18Damp} to the time -- dependent semilinear model \eqref{Semi EdeS k u}.

\subsection{Notations}

Throughout the paper we will employ the following notations: $\phi_k(t)\doteq \frac{t^{1-k}}{1-k}$ denotes a distance function produced by the speed of propagation $a_k(t)=t^{-k}$, while the amplitude of the light cone is given by the function 
\begin{align} \label{def A k}
A_k(t)\doteq \int_1^t \tau^{-k} \mathrm{d}\tau = \phi_k(t) -\phi_k(1);
\end{align} the ball with radius $R$ around the origin is denoted $B_R$; $f\lesssim g$ means that there exists a positive constant $C$ such that $f\leqslant C g$ and, similarly, for $f\gtrsim g$; $\mathrm{I}_\nu$ and $\mathrm{K}_\nu$ denote the modified Bessel function of first and second kind of order $\nu$, respectively; finally, 
\begin{align} \label{def N(k)}
N(k) \doteq \frac{1-2k+\sqrt{4k^2-4k+8}}{2(1-k)}
\end{align} denotes the threshold for the spatial dimension in determining the dominant exponent between $p_0(n,k)$ and  $p_1(n,k)$ (more specifically, $p_0(n,k)>p_1(n,k)$ if and only if $n>N(k)$, while $p_0(n,k)\leqslant p_1(n,k)$ for $n\leqslant N(k)$).

\subsection{Main results}

The main results of this work are the following blow -- up results that combined together provide a full picture of the critical case $p=\max\{p_0(n,k),p_1(n,k)\}$ for the Cauchy problem 
\begin{align}\label{Semi EdeS k} 
\begin{cases}  u_{tt} - t^{-2k}\Delta u= t^{1-p}|u|^p & x\in \mathbb{R}^n, \ t\in (1,T), \\
u(1,x)= \varepsilon u_0(x) & x\in \mathbb{R}^n, \\
 u_t(1,x)= \varepsilon u_1(x) & x\in \mathbb{R}^n,
\end{cases}
\end{align} where $p>1$, $\varepsilon>0$ is a parameter describing the size of initial data and $k\in [0,1)$.

Before stating the main results, let us introduce the notion of energy solution to the semilinear Cauchy problem \eqref{Semi EdeS k}.

\begin{definition} \label{Def energy sol} Let $u_0\in H^1(\mathbb{R}^n)$ and $u_1\in L^2(\mathbb{R}^n)$. We say that 
\begin{align*}
u\in \mathcal{C} \big([1,T), H^1(\mathbb{R}^n)\big) \cap \mathcal{C}^1 \big([1,T), L^2(\mathbb{R}^n)\big)\cap L^p_{\mathrm{loc}}\big([1,T)\times \mathbb{R}^n\big)
\end{align*} is an energy solution to \eqref{Semi EdeS k} on $[1,T)$ if $u$ fulfills $u(1,\cdot) = \varepsilon u_0$ in $H^1(\mathbb{R}^n)$ and the integral relation
\begin{align}
& \int_{\mathbb{R}^n} \partial_t u(t,x) \psi (t,x) \, \mathrm{d}x -\varepsilon \int_{\mathbb{R}^n} u_1(x) \psi (1,x) \, \mathrm{d}x 
-\int_1^t\int_{\mathbb{R}^n} \big( \partial_t u (s,x) \psi_s(s,x) -s^{-2k} \nabla u(s,x) \cdot \nabla \psi (s,x)\big) \mathrm{d}x \, \mathrm{d}s \notag \\
& \quad = \int_1^t\int_{\mathbb{R}^n} s^{1-p} |u(s,x)|^p \psi (s,x) \, \mathrm{d} x \, \mathrm{d} s \label{integral identity def energy sol}
\end{align} for any $\psi\in \mathcal{C}_0^\infty([1,T)\times \mathbb{R}^n)$ and any $t\in (1,T)$.
\end{definition}

We point out that performing a further step of integration by parts in \eqref{integral identity def energy sol}, we find the integral relation
\begin{align}
& \int_{\mathbb{R}^n} \big( \partial_t u(t,x) \psi (t,x)-u(t,x)\psi_s(t,x)\big)  \mathrm{d}x -\varepsilon \int_{\mathbb{R}^n} \big( u_1(x) \psi (1,x)- u_0(x)\psi_s(1,x)\big)  \mathrm{d}x  \notag  \\ & \qquad 
+\int_1^t\int_{\mathbb{R}^n}   u (s,x)  \big( \psi_{ss}(s,x) -s^{-2k} \Delta \psi (s,x)\big) \mathrm{d}x \, \mathrm{d}s  = \int_1^t\int_{\mathbb{R}^n} s^{1-p} |u(s,x)|^p \psi (s,x) \, \mathrm{d} x \, \mathrm{d} s \label{integral identity  weak sol}
\end{align} for any $\psi\in \mathcal{C}_0^\infty([1,T)\times \mathbb{R}^n)$ and any $t\in (1,T)$.

\begin{remark}\label{Remark support} Let us stress that if the Cauchy data are compactly supported, say $\mathrm{supp}\, u_j \subset B_R$ for $j=0,1$ and for some $R>0$, then, for any $t\in (1,T)$ a local solution $u$ to \eqref{Semi EdeS k} satisfies the support condition $$\mathrm{supp} \, u(t, \cdot) \subset B_{R+A_k(t)} ,$$ where $A_k$ is defined by \eqref{def A k}. Therefore, in Definition \ref{Def energy sol} we may also consider test functions which are not compactly supported, namely, $\psi \in \mathcal{C}^\infty([1,T)\times \mathbb{R}^n)$. 
\end{remark}

\begin{theorem} \label{Theorem critical case p0}
Let $n \in \mathbb{N}^*$ such that $n>N(k)$ and $p=p_0(n,k)$. Let us assume that $u_0\in H^1(\mathbb{R}^n)$ and $u_1\in L^2(\mathbb{R}^n)$ are nonnegative, nontrivial and compactly supported functions with supports contained in $B_R$ for some $R>0$. Let $$u\in \mathcal{C} \big([1,T), H^1(\mathbb{R}^n)\big) \cap \mathcal{C}^1 \big([1,T), L^2(\mathbb{R}^n)\big)\cap L^p_{\mathrm{loc}}\big([1,T)\times \mathbb{R}^n\big)$$ be an energy solution to \eqref{Semi EdeS k} according to Definition \ref{Def energy sol}  with lifespan $T=T(\varepsilon)$ and satisfying the support condition $\mathrm{supp} \, u(t,\cdot)\subset B_{A_k(t)+R}$ for any $t\in (1,T)$. 

Then, there exists a positive constant $\varepsilon_0=\varepsilon_0(u_0,u_1,n,p,k,R)$ such that for any $\varepsilon\in (0,\varepsilon_0]$ the energy solution $u$ blows up in finite time. Moreover, the upper bound estimate for the lifespan
\begin{align*}
T(\varepsilon)\leqslant \exp\left(C\varepsilon^{-p(p-1)}\right)
\end{align*} holds, where the constant $C>0$ is independent of $\varepsilon$.
\end{theorem}

\begin{theorem}  \label{Theorem critical case p1}
Let $n \in \mathbb{N}^*$ such that $n\leqslant N(k)$ and $p=p_1(n,k)$. Let us assume that $u_0\in H^1(\mathbb{R}^n)$ and $u_1\in L^2(\mathbb{R}^n)$ are nonnegative, nontrivial and compactly supported functions with supports contained in $B_R$ for some $R>0$. Let $$u\in \mathcal{C} \big([1,T), H^1(\mathbb{R}^n)\big) \cap \mathcal{C}^1 \big([1,T), L^2(\mathbb{R}^n)\big)\cap L^p_{\mathrm{loc}}\big([1,T)\times \mathbb{R}^n\big)$$ be an energy solution to \eqref{Semi EdeS k} according to Definition \ref{Def energy sol}  with lifespan $T=T(\varepsilon)$ and satisfying the support condition $\mathrm{supp} \, u(t,\cdot)\subset B_{A_k(t)+R}$ for any $t\in (1,T)$. 

Then, there exists a positive constant $\varepsilon_0=\varepsilon_0(u_0,u_1,n,p,k,R)$ such that for any $\varepsilon\in (0,\varepsilon_0]$ the energy solution $u$ blows up in finite time. Moreover, the upper bound estimate for the lifespan
\begin{align*}
T(\varepsilon)\leqslant \exp\left(C\varepsilon^{-(p-1)}\right)
\end{align*} holds, where the constant $C>0$ is independent of $\varepsilon$.
\end{theorem}

The remaining part of the paper is organized as follows: in Section \ref{Section critical case p0} we prove Theorem \ref{Theorem critical case p0} by using the approach introduced in \cite{WakYor18}; then, in Section \ref{Section subcritical case} we provide a complete overview on upper bound estimates for the subcritical case (cf. Proposition \ref{Proposition lifespan subcrit}), while in Section \ref{Section critical case p1}  we show the proof of Theorem \ref{Theorem critical case p1}; finally, in Appendix \ref{Appendix solutions y''-lambda^2 t^(-4/3)y=0} we provide a different proof of Proposition \ref{Proposition representations y0 and y1} in the special case of Einstein -- de Sitter spacetime.

\section[Semilinear wave equation in EdeS spacetime: critical case $p=p_0(n,k)$]{Semilinear wave equation in EdeS spacetime: 1st critical case} \label{Section critical case p0}

Our goal is to prove a blow -- up result in the critical case $p=p_0(n,k)$, where $p_0(n,k)$ is the greatest root of the quadratic equation
\begin{align}\label{equation critical exponent general case}
\left(\tfrac{n-1}{2}+\tfrac{2-k}{2(1-k)}\right)p^2-
\left(\tfrac{n+1}{2}+\tfrac{2+3k}{2(1-k)}\right)p-1=0.
\end{align} 

The approach that we will follow is based on the technique introduced in \cite{WakYor18} and subsequently applied to different wave models (cf. \cite{WakYor18Damp,PalTak19,PalTak19mix,LinTu19,ChenPal19MGT,ChenPal19SWENM}).

We are going to introduce a time -- dependent functional that depends on a local in time solution to \eqref{Semi EdeS k} and to study its blow -- up dynamic. In particular, the blow -- up result will be obtained by applying the so -- called \emph{slicing procedure} in an iteration argument to show a sequence of lower bound estimates for the above mentioned functional. 

The section is organized as follows: in Section \ref{Subsection Aux functions} we determine a pair of auxiliary functions which have a fundamental role in the definition of the time -- dependent functional and in the determination of the iteration frame, while in Section \ref{Subsection estimates auxiliary functions} we establish some fundamental estimates for these functions; then, in Section \ref{Subsection iteration frame} we establish the iteration frame for the functional and, finally, in Section \ref{Subsection iteration procedure} we prove the blow -- up result by using an iteration procedure.

\subsection{Auxiliary functions} \label{Subsection Aux functions}  

In this section, we are going to introduce two auxiliary  functions (see $\xi_q$ and $\eta_q$ below) analogously to the corresponding functions introduced in \cite{WakYor18}, which represent, in turn, a generalization of the solution to the classical free wave equation given in \cite{Zhou07}. Those auxiliary functions are defined by using the remarkable function 
\begin{align}\label{def Yordanov-Zhang function}
\varphi (x) \doteq \begin{cases} \int_{\mathbb{S}^{n-1}} \mathrm{e}^{x\cdot \omega} \mathrm{d} \sigma_\omega & \mbox{if} \ n\geqslant 2, \\ \cosh x 
& \mbox{if} \ n=1 \end{cases} 
\end{align} introduced in \cite{YZ06}. Let us recall briefly the main properties of this function: $\varphi$ is a positive and smooth function that satisfies $\Delta \varphi =\varphi$ and asymptotically behaves like $|x|^{-\frac{n-1}{2}}\mathrm{e}^{|x|}$ as $|x|\to \infty$ up to a positive multiplicative constant.

In order to introduce the definition of the auxiliary functions, let us begin by determining the solutions $y_j=y_j(t,s;\lambda,k)$, $j\in\{0,1\}$, of the non-autonomous, parameter-dependent, ordinary Cauchy problems
\begin{align}\label{CP yj(t,s;lambda,k)} 
\begin{cases}  \partial_t^2 y_j(t,s;\lambda,k) - \lambda^2 t^{-2k} y_j(t,s;\lambda,k)= 0, &  t>s, \\
y_j(s,s;\lambda,k)= \delta_{0j}, \\
 \partial_t y_j(s,s;\lambda,k)= \delta_{1j},
\end{cases}
\end{align} where $\delta_{ij}$ denotes the Kronecker delta, $s\geqslant 1$ is the initial time and $\lambda>0$ is a real parameter. 

Let us recall that we denote by $\phi_k(t)= \tfrac{t^{1-k}}{1-k}$ a primitive of the speed of propagation $a(t) = t^{-k}$ for the wave equation in  \eqref{Semi EdeS k}. In order to find a system of independent solutions to 
\begin{align}\label{equation y}
\frac{\mathrm{d}^2 y}{\mathrm{d} t^2} -\lambda^2 t^{-2k}y=0
\end{align} we perform first a change of variables. Let $\tau= \tau(t;\lambda,k)\doteq  \lambda \phi_k(t)$. Since
\begin{align*}
\frac{\mathrm{d} y}{\mathrm{d} t} &= \lambda t^{-k} \frac{\mathrm{d} y}{\mathrm{d} \tau}, \qquad \frac{\mathrm{d}^2 y}{\mathrm{d} t^2} = \lambda^2 t^{-2k} \frac{\mathrm{d}^2 y}{\mathrm{d} \tau^2}-\lambda k t^{-k-1} \frac{\mathrm{d} y}{\mathrm{d} \tau},
\end{align*} then, $y$ solves \eqref{equation y} if and only if it solves
\begin{align}\label{equation y tau}
\tau \frac{\mathrm{d}^2 y}{\mathrm{d} \tau^2} -\frac{k}{1-k} \frac{\mathrm{d} y}{\mathrm{d} \tau}-\tau y=0.
\end{align} Next, we carry out the transformation $y(\tau)=\tau^\nu w(\tau)$ with $\nu\doteq \tfrac{1}{2(1-k)}$. Therefore, $y$ solves \eqref{equation y tau} if and only if $w$ solves the modified Bessel equation of order $\nu$
\begin{align}\label{Bessel equation w}
\tau^2\frac{\mathrm{d}^2 w}{\mathrm{d} \tau^2} +\tau\frac{\mathrm{d} w}{\mathrm{d} \tau}-\left(\nu^2+\tau^2 \right) w=0,
\end{align} where we applied the straightforward relations
\begin{align*}
\frac{\mathrm{d} y}{\mathrm{d} \tau}= \nu \tau^{\nu-1} w(\tau) +\tau^\nu \frac{\mathrm{d} w}{\mathrm{d} \tau} ,\qquad \frac{\mathrm{d}^2 y}{\mathrm{d} \tau^2} = \nu (\nu-1) \tau^{\nu-2} w+2\nu \tau^{\nu-1} \frac{\mathrm{d} w}{\mathrm{d} \tau}+\tau^\nu \frac{\mathrm{d}^2 w}{\mathrm{d} \tau^2}.
\end{align*} If  we employ  as independent solutions to \eqref{Bessel equation w} the modified Bessel function of first and second kind of order $\nu$, denoted, respectively, by $\mathrm{I}_\nu(\tau)$ and $\mathrm{K}_\nu(\tau)$, then, the pair of functions
\begin{align*}
V_0(t;\lambda,k) &\doteq \tau ^\nu \mathrm{I}_\nu (\tau) = (\lambda \phi_k(t))^\nu \mathrm{I}_\nu (\lambda \phi_k(t)), \\
V_1(t;\lambda,k) & \doteq \tau ^\nu \mathrm{K}_\nu (\tau)  = (\lambda \phi_k(t))^\nu \mathrm{K}_\nu (\lambda \phi_k(t))
\end{align*} is a basis of the space of solutions to \eqref{equation y}.

\begin{proposition} \label{Proposition representations y0 and y1} The functions
\begin{align}
y_0(t,s;\lambda,k) &\doteq \lambda \left(t/s\right)^{1/2}\phi_k(s)\big[\mathrm{I}_{\nu-1}(\lambda \phi_k (s))\, \mathrm{K}_{\nu}(\lambda \phi_k (t))+\mathrm{K}_{\nu-1}(\lambda \phi_k (s))\,\mathrm{I}_{\nu}(\lambda \phi_k (t))\big],  \label{def y0(t,s;lambda,k)} \\
y_1(t,s;\lambda,k) &\doteq  (1-k)^{-1} (st)^{1/2} \big[\mathrm{K}_{\nu}(\lambda \phi_k (s))\, \mathrm{I}_{\nu}(\lambda \phi_k (t))-\mathrm{I}_{\nu}(\lambda \phi_k (s))\,\mathrm{K}_{\nu}(\lambda \phi_k (t))\big],  \label{def y1(t,s;lambda,k)} 
\end{align} solve the Cauchy problems \eqref{CP yj(t,s;lambda,k)} for $j=0$ and $j=1$, respectively, where $\nu= 1/(2(1-k))$, $\phi_k(t)= t^{1-k}/(1-k)$ and $\mathrm{I}_\nu,\mathrm{K}_\nu$ denote the modified Bessel function of order $\nu$ of the first and second kind, respectively.
\end{proposition}

\begin{proof}
We have seen that $V_0,V_1$ form a system of independent solutions to \eqref{equation y}. Therefore, we may express the solutions of \eqref{CP yj(t,s;lambda,k)} as linear combinations of $V_0,V_1$ as follows:
\begin{align} \label{representation yj with aj and bj}
y_j(t,s;\lambda,k) = a_j(s;\lambda,k) V_0(t;\lambda,k)+  b_j(s;\lambda,k) V_1(t;\lambda,k)
\end{align} for suitable coefficients $a_j(s;\lambda,k), b_j(s;\lambda,k)$, $j\in\{0,1\}$. Using the initial conditions $\partial^i_t y_j(s,s;\lambda,k)=\delta_{ij}$, we find the system
\begin{align*}
\left(\begin{array}{cc}
V_0(s;\lambda,k) & V_1(s;\lambda,k)  \\ 
\partial_t V_0(s;\lambda,k) & \partial_t V_1(s;\lambda,k) 
\end{array} \right) \left(\begin{array}{cc}
a_0(s;\lambda,k) & a_1(s;\lambda,k)  \\ 
b_0(s;\lambda,k) & b_1(s;\lambda,k) 
\end{array} \right) = I,
\end{align*} where $I$ denotes the identity matrix. So, in  order to determine the coefficients in \eqref{representation yj with aj and bj}, we have to calculate explicitly the inverse matrix
\begin{align}\label{inverse matrix}
\left(\begin{array}{cc}
V_0(s;\lambda,k) & V_1(s;\lambda,k)  \\ 
\partial_t V_0(s;\lambda,k) & \partial_t V_1(s;\lambda,k) 
\end{array} \right)^{-1} = \left(\mathcal{W}(V_0,V_1)(s;\lambda,k)\right)^{-1}\left(\begin{array}{cc}
 \partial_t V_1(s;\lambda,k)  & -V_1(s;\lambda,k)  \\ 
-\partial_t V_0(s;\lambda,k) & V_0(s;\lambda,k) 
\end{array} \right),
\end{align} where $\mathcal{W}(V_0,V_1)$ is the Wronskian of $V_0,V_1$. Clearly, we need to express in a more suitable way $\mathcal{W}(V_0,V_1)$. Let us calculate the $t$ -- derivative of $V_0,V_1$. Recalling that $\phi_k(t)=t^{1-k}/(1-k)$ and $\nu=1/(2(1-k))$, it results
\begin{align*}
\partial_t V_0(t;\lambda,k) & =\nu (\lambda\phi_k(t))^{\nu-1} \lambda \phi_k'(t) \, \mathrm{I}_\nu(\lambda \phi_k(t))+ (\lambda\phi_k(t))^{\nu} \, \mathrm{I}_\nu'(\lambda \phi_k(t))  \lambda \phi_k'(t) \\
 & =\tfrac{1}{2t} (\lambda\phi_k(t))^{\nu} \, \mathrm{I}_\nu(\lambda \phi_k(t))+ (\lambda\phi_k(t))^{\nu} (\lambda \phi_k'(t)) \, \mathrm{I}_\nu'(\lambda \phi_k(t))  
\end{align*} and, analogously, 
\begin{align*}
\partial_t V_1(t;\lambda,k) & =\tfrac{1}{2t} (\lambda\phi_k(t))^{\nu} \, \mathrm{K}_\nu(\lambda \phi_k(t))+ (\lambda\phi_k(t))^{\nu} (\lambda \phi_k'(t)) \, \mathrm{K}_\nu'(\lambda \phi_k(t)) .
\end{align*} Consequently, we can express $\mathcal{W}(V_0,V_1)$ as follows:
\begin{align*}
\mathcal{W}(V_0,V_1)(t;\lambda,k) & = (\lambda \phi_k(t))^{2\nu} (\lambda \phi_k'(t)) \big[\mathrm{K}_\nu'(\lambda \phi_k(t))\, \mathrm{I}_\nu(\lambda \phi_k(t)) -\mathrm{I}_\nu'(\lambda \phi_k(t))\, \mathrm{K}_\nu(\lambda \phi_k(t)) \big] \\
& = (\lambda \phi_k(t))^{2\nu} (\lambda \phi_k'(t)) \mathcal{W}(\mathrm{I}_\nu,\mathrm{K}_\nu) (\lambda\phi_k(t)) = -(\lambda \phi_k(t))^{2\nu-1} (\lambda \phi_k'(t))\\
& = -\lambda^{2\nu}  (\phi_k(t))^{2\nu-1}  \phi_k'(t) = -c_k^{-1} \lambda^{2\nu},
\end{align*} where $c_k\doteq (1-k)^{k/(1-k)}$ and in the third equality we used the value of the Wronskian of $\mathrm{I}_\nu,\mathrm{K}_\nu$
\begin{align*}
\mathcal{W}(\mathrm{I}_\nu,\mathrm{K}_\nu) (z)= \mathrm{I}_\nu (z) \partial_z\mathrm{K}_\nu(z)- \partial_z \mathrm{I}_\nu (z)\mathrm{K}_\nu (z) =- \frac1z.
\end{align*} Let us underline that $\mathcal{W}(V_0,V_1)(t;\lambda,k)$ does not actually depend on $t$, due to the absence of the first order term in \eqref{equation y}.

 Plugging the previous representation of  $\mathcal{W}(V_0,V_1)$ in \eqref{inverse matrix}, we get
\begin{align*}
\left(\begin{array}{cc}
a_0(s;\lambda,k) & a_1(s;\lambda,k)  \\ 
b_0(s;\lambda,k) & b_1(s;\lambda,k) 
\end{array} \right) = -c_k \lambda^{-2\nu} \left(\begin{array}{cc}
 \partial_t V_1(s;\lambda,k)  & -V_1(s;\lambda,k)  \\ 
-\partial_t V_0(s;\lambda,k) & V_0(s;\lambda,k) 
\end{array} \right).
\end{align*}
Let us begin by proving \eqref{def y0(t,s;lambda,k)}. Using the above representation of $a_0(s;\lambda,k),b_0(s;\lambda,k)$ in \eqref{representation yj with aj and bj}, we obtain
\begin{align*}
y_0(t,s;\lambda,k) &= c_k \lambda^{-2\nu} \big\{\partial_tV_0(s;\lambda,k) V_1(t;\lambda,k)-\partial_tV_1(s;\lambda,k) V_0(t;\lambda,k)\big\} \\
& = c_k \lambda^{-2\nu}(\lambda\phi_k(s))^{\nu} (\lambda\phi_k(t))^{\nu}\Big\{ \big[\tfrac{1}{2s}\,  \mathrm{I}_\nu(\lambda \phi_k(s))+ (\lambda \phi_k'(s)) \, \mathrm{I}_\nu'(\lambda \phi_k(s))  \big] \mathrm{K}_\nu(\lambda \phi_k(t)) \\
& \qquad \phantom{ c_k \lambda^{-2\nu}(\phi_k(s))^{\nu} (\phi_k(t))^{\nu}\big\{ }- \big[\tfrac{1}{2s}\,  \mathrm{K}_\nu(\lambda \phi_k(s))+ (\lambda \phi_k'(s)) \, \mathrm{K}_\nu'(\lambda \phi_k(s))  \big] \mathrm{I}_\nu(\lambda \phi_k(t)) \Big\} \\
& = c_k (\phi_k(s)\phi_k(t))^{\nu} (2s)^{-1}  \big\{ \mathrm{I}_\nu(\lambda \phi_k(s))\,  \mathrm{K}_\nu(\lambda \phi_k(t)) -\mathrm{K}_\nu(\lambda \phi_k(s))\,  \mathrm{I}_\nu(\lambda \phi_k(t)) \big\} \\
& \quad + c_k \lambda  (\phi_k(s)\phi_k(t))^{\nu} \phi_k'(s) \big\{ \mathrm{I}'_\nu(\lambda \phi_k(s))\,  \mathrm{K}_\nu(\lambda \phi_k(t)) -\mathrm{K}'_\nu(\lambda \phi_k(s))\,  \mathrm{I}_\nu(\lambda \phi_k(t)) \big\}.
\end{align*}
Applying the recursive relations for the derivatives of the modified Bessel functions
\begin{align*}
\frac{\partial \,\mathrm{I}_\nu}{\partial z}(z) & = -\frac{ \nu}{z} \, \mathrm{I}_\nu(z)+ \mathrm{I}_{\nu-1}(z), \\
\frac{\partial  \,\mathrm{K}_\nu}{\partial z}(z) & = -\frac{ \nu}{z} \, \mathrm{K}_\nu(z)- \mathrm{K}_{\nu-1}(z),
\end{align*} to the last relation, we arrive at
\begin{align}
y_0(t,s;\lambda,k) &= c_k (\phi_k(s)\phi_k(t))^{\nu} \underbrace{ \left[ (2s)^{-1} -\tfrac{\nu \lambda \phi_k'(s)}{ \lambda \phi_k(s)} \right]}_{=0}\big\{ \mathrm{I}_\nu(\lambda \phi_k(s))\,  \mathrm{K}_\nu(\lambda \phi_k(t)) -\mathrm{K}_\nu(\lambda \phi_k(s))\,  \mathrm{I}_\nu(\lambda \phi_k(t)) \big\}\notag \\
& \quad + c_k \lambda  (\phi_k(s)\phi_k(t))^{\nu} \phi_k'(s) \big\{ \mathrm{I}_{\nu-1}(\lambda \phi_k(s))\,  \mathrm{K}_\nu(\lambda \phi_k(t)) +\mathrm{K}_{\nu-1}(\lambda \phi_k(s))\,  \mathrm{I}_\nu(\lambda \phi_k(t)) \big\} \notag \\
& = c_k \lambda  (\phi_k(s)\phi_k(t))^{\nu} \phi_k'(s) \big\{ \mathrm{I}_{\nu-1}(\lambda \phi_k(s))\,  \mathrm{K}_\nu(\lambda \phi_k(t)) +\mathrm{K}_{\nu-1}(\lambda \phi_k(s))\,  \mathrm{I}_\nu(\lambda \phi_k(t)) \big\}. \label{intermediate representation y0}
\end{align} Since $c_k   (\phi_k(s)\phi_k(t))^{\nu} \phi_k'(s) = (1-k)^{-1} (st)^{1/2} s^{-k} =(t/s)^{1/2} \phi_k(s)$,  \eqref{intermediate representation y0} yields immediately \eqref{def y0(t,s;lambda,k)}. Let us prove now the representation \eqref{def y1(t,s;lambda,k)}. Plugging the above determined expressions for $a_1(s;\lambda,k),b_1(s;\lambda,k)$ in \eqref{representation yj with aj and bj}, we have
\begin{align}
y_1(t,s;\lambda,k) &= c_k \lambda^{-2\nu} \big\{V_1(s;\lambda,k) V_0(t;\lambda,k)-V_0(s;\lambda,k) V_1(t;\lambda,k)\big\} \notag \\
& = c_k \lambda^{-2\nu}(\lambda\phi_k(s))^{\nu} (\lambda\phi_k(t))^{\nu}\big\{  \mathrm{K}_\nu(\lambda \phi_k(s))   \, \mathrm{I}_\nu(\lambda \phi_k(t)) -\mathrm{I}_\nu(\lambda \phi_k(s)) \, \mathrm{K}_\nu(\lambda \phi_k(t))\big\} \notag \\
& = c_k (\phi_k(s) \phi_k(t))^{\nu}\big\{  \mathrm{K}_\nu(\lambda \phi_k(s))   \, \mathrm{I}_\nu(\lambda \phi_k(t)) -\mathrm{I}_\nu(\lambda \phi_k(s)) \, \mathrm{K}_\nu(\lambda \phi_k(t))\big\} . \label{intermediate representation y1}
\end{align} Thus, using $c_k (\phi_k(s) \phi_k(t))^{\nu} =(st)^{1/2}/(1-k)$, from \eqref{intermediate representation y1} follows \eqref{def y1(t,s;lambda,k)}. This concludes the proof.
\end{proof}

\begin{remark}  In the special case $k=2/3$, $y_0(t,s;\lambda,k)$ and $y_1(t,s;\lambda,k)$ can be expressed in terms of elementary functions. Indeed by using the explicit representations
\begin{align*}
\mathrm{I}_{\frac{1}{2}}(z) & = \sqrt{\frac{2}{\pi}}  \, \frac{\sinh z}{z^{1/2}}, \quad \mathrm{I}_{\frac{3}{2}}(z) \ = \sqrt{\frac{2}{\pi}}  \, \frac{z\cosh z -\sinh z}{z^{3/2}}, \\  \mathrm{K}_{\frac{1}{2}}(z) & = \sqrt{\frac{\pi}{2}}  \, \frac{\mathrm{e}^{-z}}{z^{1/2}}, \quad \ \ \mathrm{K}_{\frac{3}{2}}(z) = \sqrt{\frac{\pi}{2}}  \, \frac{\mathrm{e}^{-z}(z+1)}{z^{3/2}},
\end{align*} we can derive the following representations
\begin{align}
y_0(t,s;\lambda,2/3) &= \left(\frac{t}{s}\right)^{1/3} \cosh\big( 3\lambda \big(t^{1/3}-s^{1/3}\big)\big) -\frac{1}{3\lambda s^{1/3}}  \sinh\big( 3\lambda \big(t^{1/3}-s^{1/3}\big)\big), \label{def y0(t,s;lambda,2/3)} \\
y_1(t,s;\lambda,2/3) &= \left[\frac{\left(st\right)^{1/3}}{\lambda} -\frac{1}{9\lambda^3}\right] \cosh\big( 3\lambda \big(t^{1/3}-s^{1/3}\big)\big) +\frac{1}{3\lambda^2}  \big(t^{1/3}-s^{1/3}\big) \sinh\big( 3\lambda \big(t^{1/3}-s^{1/3}\big)\big).  \label{def y1(t,s;lambda,2/3)} 
\end{align} Actually, in this case it is possible to derive the representations of $y_0(t,s;\lambda,2/3),y_1(t,s;\lambda,2/3)$ by reducing \eqref{equation y} to a confluent hypergeometric  equation instead of a modified Bessel equation. For a detailed proof see Appendix \ref{Appendix solutions y''-lambda^2 t^(-4/3)y=0}.
\end{remark}

\begin{lemma} Let $y_0$, $y_1$ be the functions defined in \eqref{def y0(t,s;lambda,k)} and \eqref{def y1(t,s;lambda,k)}, respectively. Then, the following identities are satisfied for any $t\geqslant s\geqslant 1$
\begin{align}
& \frac{\partial y_1}{\partial s}(t,s;\lambda,k)= -y_0(t,s;\lambda,k), \label{dy1/ds= -y0} \\
& \frac{\partial^2 y_1}{\partial s^2}(t,s;\lambda,k) -\lambda^2 s^{-2k}y_1(t,s;\lambda,k)= 0. \label{y1 adjoiunt equation} 
\end{align} 
\end{lemma}

\begin{remark} As the operator $(\partial_t^2 -\lambda^2t^{-2k})$ is formally self-adjoint, in particular \eqref{dy1/ds= -y0} and \eqref{y1 adjoiunt equation} tell us that $y_1$ solves also the adjoint problem to \eqref{equation y} with final conditions $(0,-1)$. 
\end{remark}

\begin{proof} Let us introduce the pair of independent solutions to \eqref{equation y}
\begin{align*}
 z_0(t;\lambda,k) & \doteq y_0(t,1;\lambda,k), \\
 z_1(t;\lambda,k) & \doteq y_1(t,1;\lambda,k).
\end{align*} By standard computations, we may show the representations
\begin{align*}
y_0(t,s;\lambda,k) & =  z_1'(s;\lambda,k)  z_0(t;\lambda,k)- z_0'(s;\lambda,k)  z_1(t;\lambda,k),\\
y_1(t,s;\lambda,k) & =  z_0(s;\lambda,k)  z_1(t;\lambda,k)- z_1(s;\lambda,k)  z_0(t;\lambda,k).
\end{align*} In particular, we used that the Wronskian of $z_0,z_1$ is identically 1. First we prove \eqref{dy1/ds= -y0}. Differentiating the second one of the previous representations with respect to $s$ and then using the first one, we get immediately
\begin{align*}
\frac{\partial y_1}{\partial s}(t,s;\lambda,k) & =  z_0'(s;\lambda,k)  z_1(t;\lambda,k)- z_1'(s;\lambda,k)  z_0(t;\lambda,k) = - y_0(t,s;\lambda,k) .
\end{align*} Since $z_0,z_1$ are solutions of \eqref{equation y}, then,
\begin{align*}
 \frac{\partial^2 y_1}{\partial s^2}(t,s;\lambda,k) & = z_0''(s;\lambda,k)  z_1(t;\lambda,k)- z_1''(s;\lambda,k)  z_0(t;\lambda,k) \\  & = \lambda^2 s^{-2k} z_0(s;\lambda,k)  z_1(t;\lambda,k)- \lambda^2 s^{-2k} z_1(s;\lambda,k)  z_0(t;\lambda,k) = \lambda^2 s^{-2k} y_1(t,s;\lambda,k) .
\end{align*}So, we prove \eqref{y1 adjoiunt equation} as well.
\end{proof}

\begin{proposition} \label{Proposition integral relation test function} Let $u_0\in H^1(\mathbb{R}^n)$ and $u_1\in L^2(\mathbb{R}^n)$ be functions such that $\mathrm{supp}\, u_j \subset B_R$ for $j=0,1$ and for some $R>0$ and let $\lambda>0$ be a parameter. Let $u$ be a local in time energy solution to \eqref{Semi EdeS k} on $[1,T)$ according to Definition \ref{Def energy sol}. Then, the following integral identity is satisfied for any $t\in [1,T)$
\begin{align} 
\int_{\mathbb{R}^n} u(t,x) \varphi_\lambda (x) \, \mathrm{d}x & = \varepsilon \, y_0(t,1;\lambda,k) \int_{\mathbb{R}^n} u_0(x) \varphi_\lambda(x) \, \mathrm{d}x + \varepsilon \, y_1(t,1;\lambda,k) \int_{\mathbb{R}^n} u_1(x) \varphi_\lambda(x) \, \mathrm{d}x \notag \\ & \quad +\int_1^t s^{1-p} y_1(t,s;\lambda,k)  \int_{\mathbb{R}^n} |u(s,x)|^p \varphi_\lambda (x) \, \mathrm{d}x \, \mathrm{d}s, \label{fundametal integral equality}
\end{align} where $\varphi_\lambda(x)\doteq \varphi(\lambda x)$ and $\varphi$ is defined by \eqref{def Yordanov-Zhang function}.
\end{proposition}

\begin{proof} Since we assumed $u_0,u_1$ compactly supported, we may consider a test function $\psi\in \mathcal{C}^\infty([1,T)\times\mathbb{R}^n)$ in Definition \ref{Def energy sol} according to Remark \ref{Remark support}. Therefore, we consider $\psi(s,x)=y_1(t,s;\lambda,k)\varphi_\lambda(x)$ (here $t,\lambda$ can be considered fixed parameters). Hence, $\psi$ satisfies
\begin{align*}
\psi(t,x) &=y_1(t,t;\lambda,k) \varphi_\lambda(x)=0, \quad \psi_s(t,x) = \partial_s y_1(t,t;\lambda,k) \varphi_\lambda(x) =-  y_0(t,t;\lambda,k) \varphi_\lambda(x) =- \varphi_\lambda(x), \\
\psi(1,x) &=y_1(t,1;\lambda,k) \varphi_\lambda(x), \phantom{=0} \quad \psi_s(1,x) =  \partial_s y_1(t,1;\lambda,k) \varphi_\lambda(x) =-  y_0(t,1;\lambda,k) \varphi_\lambda(x), 
\end{align*}  and
\begin{align*}
\psi_{ss}(s,x) -s^{-2k} \Delta \psi(s,x) =\left( \partial_s^2 y_1(t,s;\lambda,k) -\lambda^2 s^{-2k} y_1(t,s;\lambda,k) \right) \varphi_\lambda(x)=0,
\end{align*}
where we used \eqref{dy1/ds= -y0}, \eqref{y1 adjoiunt equation} and the property $\Delta \varphi=\varphi$. 

Hence, employing this $\psi$ in \eqref{integral identity  weak sol}, we find immediately \eqref{fundametal integral equality}.
\end{proof}

\begin{proposition} \label{Proposition lower bound estimates y0 and y1} Let $y_0$, $y_1$ be the functions defined in \eqref{def y0(t,s;lambda,k)} and \eqref{def y1(t,s;lambda,k)}, respectively. Then, the following estimates are satisfied for any $t\geqslant s\geqslant 1$
\begin{align}
& y_0(t,s;\lambda,k)\geqslant  \cosh \big(\lambda (\phi_k(t)-\phi_k(s))\big) , \label{lower bound estimate y0(t,s;lambda,k)} \\
& y_1(t,s;\lambda,k)\geqslant   (st)^{\frac{k}{2}} \frac{\sinh \big(\lambda (\phi_k(t)-\phi_k(s))\big) }{\lambda}.\label{lower bound estimate y_1(t,s;lambda,k)}
\end{align} 
\end{proposition}

\begin{proof} The proof of the inequalities \eqref{lower bound estimate y0(t,s;lambda,k)} and \eqref{lower bound estimate y_1(t,s;lambda,k)} is based on the following minimum type principle: \\
\emph{ let $w=w(t,s;\lambda,k)$ be a solution of the Cauchy problem }
\begin{align}\label{CP y}
\begin{cases} \partial_t^2 w -\lambda^2 t^{-2k} w =h, &  \mbox{for} \ t>s \geqslant 1, \\  w(s)=
w_0, \ \partial_t w(s)=w_1, \end{cases}
\end{align} \emph{where $h=h(t,s;\lambda,k)$ is a continuous function; if $h\geqslant 0$ and $w_0=w_1=0$ (i.e. $w$ is a \emph{supersolution} of the homogeneous problem with trivial initial conditions), then, $w(t,s;\lambda,k)\geqslant 0$ for any $t>s$}.

 In order to prove this minimum principle, we apply the continuous dependence on initial conditions (note that for $t\geqslant 1$ the function $t^{-2k}$ is smooth). Indeed, if we denote by $w_\epsilon$ the solution to \eqref{CP y} with $w_0=\epsilon>0$ and $w_1=0$, then, $w_\epsilon$ solves the integral equation
\begin{align*}
w_\epsilon(t,s;\lambda,k) = \epsilon +\int_s^t\int_s^\tau \big(\lambda^2\sigma^{-2k} w_\epsilon(\sigma,s;\lambda,k)+h(\sigma,s;\lambda,k)\big) \mathrm{d}\sigma\, \mathrm{d}\tau.
\end{align*} By contradiction, one can prove easily that $w_\epsilon(t,s;\lambda,k)>0$ for any $t>s$. Hence, by the continuous dependence on initial data, letting $\epsilon\to 0$, we find that  $w(t,s;\lambda,k)\geqslant 0$ for any $t>s$.

 Note that if $w_0,w_1\geqslant 0$ and $w_0+w_1>0$, then, the positivity of $w$ follows straightforwardly from the corresponding integral equation via a contradiction argument, rather than working with the family $\{w_\epsilon\}_{\epsilon>0}$. Nevertheless, in what follows we consider exactly the limit case $w_0=w_1=0$, for this reason the previous digression was necessary.

Let us prove the validity of \eqref{lower bound estimate y_1(t,s;lambda,k)}. We denote by $w_1=w_1(t,s;\lambda,k)$ the function on the right -- hand side of \eqref{lower bound estimate y_1(t,s;lambda,k)}. It is easy to see that $w_1(s,s;\lambda,k)=0$ and $\partial_t w_1(s,s;\lambda,k)=1$. Moreover,
\begin{align*}
\partial_t^2 w_1(t,s;\lambda,k) &= \lambda^{-1} s^{\frac{k}{2}} \Big[ \tfrac k2 \left(\tfrac k2 -1\right) t^{\frac{k}{2}-2} \sinh \big(\lambda (\phi_k(t)-\phi_k(s))\big)+ k t^{\frac k2 -1 } \cosh \big(\lambda (\phi_k(t)-\phi_k(s))\big) \lambda \phi'_k(t) \\ 
  & \qquad  \phantom{\lambda^{-1} s^{\frac{k}{2}} \Big[}  +t^{\frac{k}{2}} \sinh \big(\lambda (\phi_k(t)-\phi_k(s))\big) (\lambda \phi'_k(t))^2+t^{\frac{k}{2}} \cosh \big(\lambda (\phi_k(t)-\phi_k(s))\big) \lambda \phi''_k(t)  \Big] \\
& = \lambda^{-1} s^{\frac{k}{2}} \Big[ \tfrac k2 \left(\tfrac k2 -1\right) t^{\frac{k}{2}-2} \sinh \big(\lambda (\phi_k(t)-\phi_k(s))\big) +t^{\frac{k}{2}} \sinh \big(\lambda (\phi_k(t)-\phi_k(s))\big) (\lambda t^{-k})^2 \Big] \\
& \leqslant \lambda^{-1} (s t)^{\frac{k}{2}} \sinh \big(\lambda (\phi_k(t)-\phi_k(s))\big) (\lambda t^{-k})^2 = \lambda^2 t^{-2k} w_1(t,s;\lambda,k). 
\end{align*}
Therefore, $y_1-w_1$ is a supersolution of \eqref{CP y} with $h=0$ and $w_0=w_1=0$. Thus, applying the minimum principle we have that $(y_1-w_1)(t,s;\lambda,k)\geqslant 0$  for any $t>s$, that is, we showed \eqref{lower bound estimate y_1(t,s;lambda,k)}.

In a completely analogous way, one can prove \eqref{lower bound estimate y0(t,s;lambda,k)}, repeating the previous argument based on the minimum principle with  $w_0(t,s;\lambda,k)\doteq  \cosh \big(\lambda (\phi_k(t)-\phi_k(s))\big)$ in place of $w_1(t,s;\lambda,k)$ and $y_0$ in place of $y_1$, respectively.
\end{proof}

After the preliminary results that we have proved so far in this section, we can now introduce the definition of the following \emph{auxiliary function}
\begin{align}
\xi_q(t,s,x;k) &\doteq \int_0^{\lambda_0} \mathrm{e}^{-\lambda (A_k(t)+R)} \cosh \big(\lambda (\phi_k(t)-\phi_k(s))\big) \varphi_\lambda(x) \lambda^q \,\mathrm{d}\lambda \label{def xi q}, \\
\eta_q(t,s,x;k) & \doteq (st)^{k/2}\int_0^{\lambda_0} \mathrm{e}^{-\lambda (A_k(t)+R)} \frac{\sinh \big(\lambda (\phi_k(t)-\phi_k(s))\big) }{\lambda  (\phi_k(t)-\phi_k(s)) }\, \varphi_\lambda(x) \lambda^q \,\mathrm{d}\lambda\label{def eta q},
\end{align} where $q>-1$, $\lambda_0>0$ is a fixed parameter and $A_k$ is defined by \eqref{def A k}.

\begin{remark} For $k=0$ the functions $\xi_q$ and $\eta_q$ coincide with the corresponding ones given in \cite{WakYor18}, provided of course that we shift the initial time in the Cauchy problem from 0 to 1.
\end{remark}

Combining the results from Propositions \ref{Proposition integral relation test function} and \ref{Proposition lower bound estimates y0 and y1}, we may finally derive a fundamental inequality, whose role will be crucial in the next sections in order to prove the blow -- up result.

\begin{corollary} \label{Corollary fund ineq}  Let $u_0\in H^1(\mathbb{R}^n)$ and $u_1\in L^2(\mathbb{R}^n)$ such that $\mathrm{supp}\, u_j \subset B_R$ for $j=0,1$ and for some $R>0$. Let $u$ be a local in time energy solution to \eqref{Semi EdeS k} on $[1,T)$ according to Definition \ref{Def energy sol}. Let $q>-1$ and let $\xi_q(t,s,x;k),\eta_q(t,s,x;k)$ be the functions defined by \eqref{def xi q} and \eqref{def eta q}, respectively. Then,
\begin{align}
\int_{\mathbb{R}^n} u(t,x) \, \xi_q(t,t,x;k) \, \mathrm{d}x & \geqslant \varepsilon \int_{\mathbb{R}^n} u_0(x) \,  \xi_q(t,1,x;k)  \, \mathrm{d}x 
+ \varepsilon \,  (\phi_k(t)-\phi_k(1)) \int_{\mathbb{R}^n} u_1(x) \, \eta_q(t,s,x;k)  \, \mathrm{d}x \notag \\ & \quad +\int_1^t  (\phi_k(t)-\phi_k(s)) \, s^{1-p}  \int_{\mathbb{R}^n} |u(s,x)|^p \eta_q(t,s,x;k) \, \mathrm{d}x \, \mathrm{d}s \label{fundamental inequality functional mathcalU}
\end{align} for any $t\in [1,T)$.
\end{corollary}

\begin{proof}
Combining \eqref{fundametal integral equality} and the lower bound estimates \eqref{lower bound estimate y0(t,s;lambda,k)}, \eqref{lower bound estimate y_1(t,s;lambda,k)}, we find
\begin{align*} 
\int_{\mathbb{R}^n} u(t,x) \varphi_\lambda (x) \, \mathrm{d}x & \geqslant \varepsilon \, \cosh \big(\lambda (\phi_k(t)-\phi_k(1))\big) \int_{\mathbb{R}^n} u_0(x) \varphi_\lambda(x) \, \mathrm{d}x \\ 
& \quad + \varepsilon \, t^{\frac{k}{2}} \lambda^{-1} \sinh \big(\lambda (\phi_k(t)-\phi_k(1))\big) \int_{\mathbb{R}^n} u_1(x) \varphi_\lambda(x) \, \mathrm{d}x  \\ & \quad +\int_1^t s^{1-p} (st)^{\frac{k}{2}} \lambda^{-1} \sinh \big(\lambda (\phi_k(t)-\phi_k(s))\big)  \int_{\mathbb{R}^n} |u(s,x)|^p \varphi_\lambda (x) \, \mathrm{d}x \, \mathrm{d}s. 
\end{align*} Multiplying both sides of the previous identity by $\mathrm{e}^{-\lambda (A_k(t)+R)}\lambda^q$, integrating with respect to $\lambda$ over $[0,\lambda_0]$ and applying Fubini's theorem, we get \eqref{fundamental inequality functional mathcalU}.
\end{proof}

\subsection{Properties of the auxiliary functions} \label{Subsection estimates auxiliary functions}

In this section, we determine some lower and upper bound estimates for the auxiliary functions $\xi_q,\eta_q$ under suitable assumptions on $q$.

Let us begin with the lower bound estimates
\begin{lemma} \label{Lemma lower bound estimates xi q and eta q} Let $n\geqslant 1$ and $\lambda_0>0$. If we assume  $q>-1$, then, for $t\geqslant s\geqslant 1$ and $|x|\leqslant A_k(s) +R$ the following lower bound estimates hold:
\begin{align}
\xi_q (t,s,x;k) &\geqslant B_0 \langle A_k(s)\rangle ^{-q-1}; \label{lower bound xi q}\\
\eta_q (t,s,x;k) & \geqslant B_1 (st)^{\frac{k}{2}} \langle A_k(t)\rangle ^{-1}\langle A_k(s)\rangle ^{-q}. \label{lower bound eta q}
\end{align}
Here $B_0,B_1$ are positive constants depending only on $\lambda_0,q,R,k$ and we employ the notation $\langle y\rangle \doteq 3+|s|$.
\end{lemma}

\begin{proof} We follow the main ideas of the proof of Lemma 3.1 in \cite{WakYor18}. Since 
\begin{align} \label{asymptotic YordanovZhang function}
\langle |x| \rangle^{-\frac{n-1}{2}}\mathrm{e}^{|x|}\lesssim \varphi(x) \lesssim\langle |x| \rangle^{-\frac{n-1}{2}}\mathrm{e}^{|x|} 
\end{align} holds for any $x\in \mathbb{R}^n$, we can find a constant $B=B(\lambda_0,R,k)>0$ independent of $\lambda$ and $s$ such that
$$B\leqslant \inf_{\lambda \in \left[\frac{\lambda_0}{\langle A_k(s)\rangle}, \frac{2\lambda_0}{\langle A_k(s)\rangle}\right]} \inf_{|x|\leqslant A_k(s)+R} \mathrm{e}^{-\lambda(A_k(s)+R)}\varphi_\lambda(x).$$
Let us begin with \eqref{lower bound xi q}. Shrinking the domain of integration in \eqref{def xi q} to $\left[\frac{\lambda_0}{\langle A_k(s)\rangle}, \frac{2\lambda_0}{\langle A_k(s)\rangle}\right]$ and applying the previous inequality, we get
\begin{align*}
\xi_q (t,s,x;k) &\geqslant  \int_{\lambda_0/\langle A_k(s)\rangle} ^{2\lambda_0/\langle A_k(s)\rangle} \mathrm{e}^{-\lambda (A_k(t)-A_k(s))} \cosh \big(\lambda (\phi_k(t)-\phi_k(s))\big) \mathrm{e}^{-\lambda (A_k(s)+R)}  \varphi_\lambda(x) \lambda^q \,\mathrm{d}\lambda \\
&\geqslant B \int_{\lambda_0/\langle A_k(s)\rangle} ^{2\lambda_0/\langle A_k(s)\rangle} \mathrm{e}^{-\lambda (A_k(t)-A_k(s))} \cosh \big(\lambda (\phi_k(t)-\phi_k(s))\big)  \lambda^q \,\mathrm{d}\lambda \\
& = B/2 \int_{\lambda_0/\langle A_k(s)\rangle} ^{2\lambda_0/\langle A_k(s)\rangle} \left(1+\mathrm{e}^{-2\lambda (\phi_k(t)-\phi_k(s))}\right)  \lambda^q \,\mathrm{d}\lambda \\ & \geqslant B/2 \int_{\lambda_0/\langle A_k(s)\rangle} ^{2\lambda_0/\langle A_k(s)\rangle} \lambda^q \,\mathrm{d}\lambda =  \frac{B(2^{q+1}-1) \lambda_0^{q+1}}{2(q+1)}  \langle A_k(s)\rangle^{-q-1}.
\end{align*} We prove now \eqref{lower bound eta q}. Repeating similar steps as before, we arrive at
\begin{align*}
\eta_q (t,s,x;k) &\geqslant (st)^{\frac{k}{2}} \int_{\lambda_0/\langle A_k(s)\rangle} ^{2\lambda_0/\langle A_k(s)\rangle} \mathrm{e}^{-\lambda (A_k(t)-A_k(s))} \frac{\sinh \big(\lambda (\phi_k(t)-\phi_k(s))\big) }{\lambda  (\phi_k(t)-\phi_k(s)) }\, \mathrm{e}^{-\lambda (A_k(s)+R)}  \varphi_\lambda(x) \lambda^q \,\mathrm{d}\lambda \\
&\geqslant \tfrac{B}{2}  (st)^{\frac{k}{2}} \int_{\lambda_0/\langle A_k(s)\rangle} ^{2\lambda_0/\langle A_k(s)\rangle}  \frac{1-\mathrm{e}^{-2\lambda (\phi_k(t)-\phi_k(s))}}{\phi_k(t)-\phi_k(s) }\,  \lambda^{q-1} \,\mathrm{d}\lambda  \\
&\geqslant \tfrac{B}{2} (st)^{\frac{k}{2}} \frac{1-\mathrm{e}^{-2\lambda_0  \frac{\phi_k(t)-\phi_k(s)}{\langle A_k(s)\rangle}}}{\phi_k(t)-\phi_k(s) } \int_{\lambda_0/\langle A_k(s)\rangle} ^{2\lambda_0/\langle A_k(s)\rangle}    \lambda^{q-1} \,\mathrm{d}\lambda \\
& = \frac{B(2^q-1)\lambda_0^q}{2q} \,   (st)^{\frac{k}{2}} \langle A_k(s)\rangle^{-q} \, \frac{1-\mathrm{e}^{-2\lambda_0  \frac{\phi_k(t)-\phi_k(s)}{\langle A_k(s)\rangle}}}{\phi_k(t)-\phi_k(s) } .
\end{align*} The previous inequality implies \eqref{lower bound eta q}, provided that $$\frac{1-\mathrm{e}^{-2\lambda_0  \frac{\phi_k(t)-\phi_k(s)}{\langle A_k(s)\rangle}}}{\phi_k(t)-\phi_k(s) } \gtrsim \langle A_k(t)\rangle^{-1}  $$ holds. Let us prove this last inequality. For $\phi_k(t)-\phi_k(s)\geqslant \frac{1}{2\lambda_0}\langle A_k(s)\rangle $, we have $$1-\mathrm{e}^{-2\lambda_0  \frac{\phi_k(t)-\phi_k(s)}{\langle A_k(s)\rangle}} \geqslant 1-\mathrm{e}^{-1}$$ and, consequently,
\begin{align*}
\frac{1-\mathrm{e}^{-2\lambda_0  \frac{\phi_k(t)-\phi_k(s)}{\langle A_k(s)\rangle}}}{\phi_k(t)-\phi_k(s) }  & \gtrsim \big(\phi_k(t)-\phi_k(s) \big)^{-1} \geqslant A_k(t)^{-1} \geqslant \langle A_k(t)\rangle^{-1}.
\end{align*} On the other hand, in the case $\phi_k(t)-\phi_k(s)\leqslant \frac{1}{2\lambda_0}\langle A_k(s)\rangle $, employing the inequality $1-\mathrm{e}^{-\sigma}\geqslant \sigma/2$ for $\sigma\in [0,1]$, we find immediately
\begin{align*}
\frac{1-\mathrm{e}^{-2\lambda_0  \frac{\phi_k(t)-\phi_k(s)}{\langle A_k(s)\rangle}}}{\phi_k(t)-\phi_k(s) }  & \geqslant  \frac{\lambda_0}{\langle A_k(s)\rangle} \geqslant  \frac{\lambda_0}{\langle A_k(t)\rangle}.
\end{align*} So, also the proof of \eqref{lower bound eta q} is completed.
\end{proof}

Next we prove an upper bound estimate in the special case $s=t$.
\begin{lemma} \label{Lemma lupper bound estimate xi q t=s} Let $n\geqslant 1$ and $\lambda_0>0$. If we assume  $q> (n-3)/2$, then, for $t\geqslant 1$ and $|x|\leqslant A_k(t) +R$ the following  upper bound estimate holds:
\begin{align}
\xi_q (t,t,x;k) &\leqslant B_2 \langle A_k(t)\rangle ^{-\frac{n-1}{2}}  \langle A_k(t) - |x|\rangle ^{\frac{n-3}{2}-q}. \label{upper bound xi q t=s}
\end{align}
Here $B_2$ is a positive constant depending only on $\lambda_0,q,R,k$ and $\langle y\rangle$ denotes the same function as in the statement of Lemma \ref{Lemma lower bound estimates xi q and eta q}.
\end{lemma}

\begin{proof} We follow the proof of Lemma 3.1 (iii) in \cite{WakYor18}. Applying \eqref{asymptotic YordanovZhang function}, we get 
\begin{align*}
\xi_q(t,t,x;k) & = \int_0^{\lambda_0} \mathrm{e}^{-\lambda (A_k(t)+R)}  \varphi_\lambda(x) \lambda^q \,\mathrm{d}\lambda \lesssim \int_0^{\lambda_0} \langle \lambda |x| \rangle^{-\frac{n-1}{2}} \mathrm{e}^{-\lambda (A_k(t)+R-|x|)}   \lambda^q \,\mathrm{d}\lambda. 
\end{align*} Let us consider separately two different cases. If $|x|\leqslant (A_k(t)+R)/2$, then,
\begin{align*}
\xi_q(t,t,x;k) &  \lesssim \int_0^{\lambda_0}  \mathrm{e}^{-\lambda (A_k(t)+R-|x|)}   \lambda^q \,\mathrm{d}\lambda  \lesssim \int_0^{\lambda_0}  \mathrm{e}^{-\lambda (A_k(t)+R)/2}   \lambda^q \,\mathrm{d}\lambda \\
&  \lesssim  (A_k(t)+R)^{-q-1} \int_0^{\infty}  \mathrm{e}^{-\mu/2}   \mu^q \,\mathrm{d}\mu \lesssim  (A_k(t)+R)^{-q-1}  
 \lesssim  \langle A_k(t)\rangle ^{-q-1} \\
 & \lesssim \langle A_k(t)\rangle ^{-\frac{n-1}{2}}  \langle A_k(t) - |x|\rangle ^{\frac{n-3}{2}-q}.
\end{align*} In particular, in the last estimate we used the inequality $\langle A_k(t) - |x|\rangle \lesssim \langle A_k(t)\rangle$, which follows trivially from $|A_k(t) - |x|| \leqslant A_k(t)$ for $ |x|\leqslant A_k(t)$ and from  $\langle A_k(t) - |x|\rangle \lesssim 1$ for $A_k(t)\leqslant |x|\leqslant (A_k(t)+R)/2$.

On the other hand, for $|x|\geqslant (A_k(t)+R)/2$, we may estimate
\begin{align}
\xi_q(t,t,x;k) & \lesssim (A_k(t)+R)^{-\frac{n-1}{2}}  \int_0^{\lambda_0} \mathrm{e}^{-\lambda (A_k(t)+R-|x|)}   \lambda^{q-\frac{n-1}{2}} \,\mathrm{d}\lambda \notag \\
 & \lesssim \langle A_k(t)\rangle ^{-\frac{n-1}{2}} (A_k(t)+R-|x|)^{-q+\frac{n-3}{2}} \int_0^{\infty} \mathrm{e}^{-\mu }   \mu^{q-\frac{n-1}{2}} \,\mathrm{d}\mu  \notag \\
 & \lesssim \langle A_k(t)\rangle ^{-\frac{n-1}{2}} (A_k(t)+R-|x|)^{-q+\frac{n-3}{2}} . \label{upper bound xiq(t,t,x) intermediate}
\end{align} When $(A_k(t)+R)/2\leqslant |x|\leqslant A_k(t)$, thanks to the inequality $A_k(t)+R-|x|\gtrsim \langle A_k(t)-|x|\rangle$, from \eqref{upper bound xiq(t,t,x) intermediate} it follows easily \eqref{upper bound xi q t=s}; while for $A_k(t)\leqslant |x|\leqslant A_k(t)+R$, as $ \langle A_k(t)-|x|\rangle\approx 1$, the estimate
\begin{align*}
\xi_q(t,t,x;k) & \lesssim (A_k(t)+R)^{-\frac{n-1}{2}}  \int_0^{\lambda_0} \mathrm{e}^{-\lambda (A_k(t)+R-|x|)}   \lambda^{q-\frac{n-1}{2}} \,\mathrm{d}\lambda \\&   \lesssim \langle A_k(t)\rangle ^{-\frac{n-1}{2}}   \int_0^{\lambda_0}   \lambda^{q-\frac{n-1}{2}} \,\mathrm{d}\lambda \lesssim \langle A_k(t)\rangle ^{-\frac{n-1}{2}}  
\end{align*} is sufficient to conclude \eqref{upper bound xi q t=s}. This completes the proof.
\end{proof}

\subsection{Derivation of the iteration frame} \label{Subsection iteration frame}

In this section, we introduce the time -- dependent functional whose dynamic is studied in order to prove the blow -- up result. Hence, we derive the iteration frame for this functional and a first lower bound estimate of logarithmic type.

Let us introduce the functional
\begin{align}\label{def functional mathcalU}
\mathcal{U}(t) \doteq t^{-\frac{k}{2}} \int_{\mathbb{R}^n} u(t,x) \, \xi_q(t,t,x;k) \, \mathrm{d} x
\end{align} for $t\geqslant 1$ and for some $q>(n-3)/2$. 
From \eqref{fundamental inequality functional mathcalU}, \eqref{lower bound xi q} and \eqref{lower bound eta q}, it follows
\begin{align*}
\mathcal{U}(t) \gtrsim B_0 \varepsilon \, t^{-\frac{k}{2}} \int_{\mathbb{R}^n} u_0(x) \, \mathrm{d}x + B_1 \varepsilon \, \frac{A_k(t)}{\langle A_k(t)\rangle } \int_{\mathbb{R}^n} u_1(x) \, \mathrm{d}x.
\end{align*} If we assume that $u_0,u_1$ are both nonnegative and nontrivial, then, we find that 
\begin{align}\label{mathcalU > epsilon}
\mathcal{U}(t)\gtrsim \varepsilon
\end{align} for any $t\in [1,T)$, where the unexpressed multiplicative constant depends on $u_0,u_1$.

In the next proposition, we derive the iteration frame for the functional $\mathcal{U}$.

\begin{proposition} \label{Proposition iteration frame} Suppose that the assumptions in Corollary \ref{Corollary fund ineq} are satisfied and let $q=(n-1)/2-1/p$. If $\mathcal{U}$ is defined by \eqref{def functional mathcalU}, then, there exists a positive constant $C=C(n,p,R,k)$ such that
\begin{align} \label{iteration frame} 
\mathcal{U}(t)\geqslant C \langle A_k(t)\rangle^{-1}\int_1^t \frac{\phi_k(t)-\phi_k(s)}{s} \big(\log \langle A_k(s)\rangle\big)^{-(p-1)} (\mathcal{U}(s))^p\, \mathrm{d}s
\end{align} for any $t\in (1,T)$. 
\end{proposition}

\begin{proof}
By the definition of the functional \eqref{def functional mathcalU}, applying H\"older's inequality we get
\begin{align*}
s^{\frac k2} \mathcal{U} (s) \leq \left(\int_{\mathbb{R}^n}|u(s,x)|^p \eta_q(t,s,x;k) \, \mathrm{d}x\right)^{1/p} \left(\int_{B_{R+A_k(s)}} \frac{\big(\xi_q(s,s,x;k)\big)^{p'}}{\big(\eta_q(t,s,x;k)\big)^{p'/p}} \, \mathrm{d}x\right)^{1/p'},
\end{align*} where $1/p+1/p'=1$. Therefore,
\begin{align} \label{intermediate lower bound int |u|^p eta_q}
\int_{\mathbb{R}^n}|u(s,x)|^p \eta_q(t,s,x;k) \, \mathrm{d}x \geqslant \big(s^{\frac k2} \mathcal{U} (s)\big)^p \left(\int_{B_{R+A_k(s)}} \frac{\big(\xi_q(s,s,x;k)\big)^{p/(p-1)}}{\big(\eta_q(t,s,x;k)\big)^{1/(p-1)}} \, \mathrm{d}x\right)^{-(p-1)}.
\end{align} Let us determine now an upper bound estimates for the integral on the right -- hand side of \eqref{intermediate lower bound int |u|^p eta_q}. By using \eqref{upper bound xi q t=s} and \eqref{lower bound eta q}, we obtain
\begin{align*}
& \int_{B_{R+A_k(s)}}  \frac{\big(\xi_q(s,s,x;k)\big)^{p/(p-1)}}{\big(\eta_q(t,s,x;k)\big)^{1/(p-1)}} \, \mathrm{d}x \\ 
& \qquad \leqslant   B_1^{-\frac{1}{p-1}} B_2^{\frac{p}{p-1}}  \langle A_k(s)\rangle ^{-\frac{n-1}{2}\frac{p}{p-1} } (st)^{-\frac{k}{2(p-1)}} \langle A_k(t)\rangle ^{\frac{1}{p-1}}\langle A_k(s)\rangle ^{\frac{q}{p-1}}\int_{B_{R+A_k(s)}}  \langle A_k(s) - |x|\rangle ^{(\frac{n-3}{2}-q)\frac{p}{p-1}}  \mathrm{d}x \\
 & \qquad \leqslant   B_1^{-\frac{1}{p-1}} B_2^{\frac{p}{p-1}}   (st)^{-\frac{k}{2(p-1)}} \langle A_k(t)\rangle ^{\frac{1}{p-1}}\langle A_k(s)\rangle ^{\frac{1}{p-1}(-\frac{n-1}{2}p+\frac{n-1}{2}-\frac{1}{p})}\int_{B_{R+A_k(s)}}  \langle A_k(s) - |x|\rangle ^{-1} \mathrm{d}x \\
 & \qquad \leqslant   B_1^{-\frac{1}{p-1}} B_2^{\frac{p}{p-1}}   (st)^{-\frac{k}{2(p-1)}} \langle A_k(t)\rangle ^{\frac{1}{p-1}}\langle A_k(s)\rangle ^{\frac{1}{p-1}(-\frac{n-1}{2}p+\frac{n-1}{2}-\frac{1}{p})+n-1}  \log \langle A_k(s) \rangle  ,
\end{align*} where in the second step we used 
 $q=(n-1)/2-1/p$ to get exactly $-1$ as power of the function in the integral.
Hence, from \eqref{intermediate lower bound int |u|^p eta_q} we get 
\begin{align*}
\int_{\mathbb{R}^n}|u(s,x)|^p \eta_q(t,s,x;k) \, \mathrm{d}x &\gtrsim \big(s^{\frac k2} \mathcal{U} (s)\big)^p (st)^{\frac{k}{2}} \langle A_k(t)\rangle ^{-1}\langle A_k(s)\rangle ^{\frac{n-1}{2}(p-1)+\frac{1}{p}-(n-1)(p-1)}  \big(\log \langle A_k(s) \rangle\big)^{-(p-1)} \\
& \gtrsim t^{\frac{k}{2}}   \langle A_k(t)\rangle ^{-1} s^{\frac {k}{2}(p+1)} \langle A_k(s)\rangle ^{-\frac{n-1}{2}(p-1)+\frac{1}{p}}    \big(\log \langle A_k(s) \rangle\big)^{-(p-1)} \big(\mathcal{U} (s)\big)^p.
\end{align*}
If we combine the previous lower bound estimate and \eqref{fundamental inequality functional mathcalU}, we have
\begin{align*}
\mathcal{U}(t) & \geqslant t^{-\frac k2} \int_1^t  (\phi_k(t)-\phi_k(s)) \, s^{1-p}  \int_{\mathbb{R}^n} |u(s,x)|^p \eta_q(t,s,x;k) \, \mathrm{d}x \, \mathrm{d}s \\
& \gtrsim    \langle A_k(t)\rangle ^{-1}  \int_1^t  (\phi_k(t)-\phi_k(s)) \, s^{1-p+\frac {k}{2}(p+1)} \langle A_k(s)\rangle ^{-\frac{n-1}{2}(p-1)+\frac{1}{p}}    \big(\log \langle A_k(s) \rangle\big)^{-(p-1)} \big(\mathcal{U} (s)\big)^p\, \mathrm{d}s \\
& \gtrsim    \langle A_k(t)\rangle ^{-1}  \int_1^t  (\phi_k(t)-\phi_k(s))  \langle A_k(s)\rangle ^{\frac{1-p}{1-k}+\frac {k(p+1)}{2(1-k)}-\frac{n-1}{2}(p-1)+\frac{1}{p}}    \big(\log \langle A_k(s) \rangle\big)^{-(p-1)} \big(\mathcal{U} (s)\big)^p\, \mathrm{d}s\\
& \gtrsim    \langle A_k(t)\rangle ^{-1}  \int_1^t  (\phi_k(t)-\phi_k(s))  \langle A_k(s)\rangle ^{-\left(\frac{n-1}{2}+\frac{2-k}{2(1-k)}\right)p+\left(\frac{n-1}{2}+\frac{2+k}{2(1-k)}\right)+\frac 1p}    \big(\log \langle A_k(s) \rangle\big)^{-(p-1)} \big(\mathcal{U} (s)\big)^p\, \mathrm{d}s,
\end{align*} where in third step we used $s= (1-k)^{\frac{1}{1-k}}(A_k(s)+\phi_k(1))^{\frac{1}{1-k}}\approx \langle A_k(s)\rangle^{\frac{1}{1-k}}$ for $s\geqslant 1$. Since $p=p_0(n,k)$ from \eqref{equation critical exponent general case} it follows
\begin{align}\label{equation critical exponent intermediate}
-\left(\tfrac{n-1}{2}+\tfrac{2-k}{2(1-k)}\right)p+\left(\tfrac{n-1}{2}+\tfrac{2+k}{2(1-k)}\right)+\tfrac 1p= -1-\tfrac{k}{1-k}=-\tfrac{1}{1-k},
\end{align} then, plugging \eqref{equation critical exponent intermediate} in the last lower bound estimate for $\mathcal{U}(t)$ we find
\begin{align*}
\mathcal{U}(t) &  \gtrsim    \langle A_k(t)\rangle ^{-1}  \int_1^t  (\phi_k(t)-\phi_k(s))  \langle A_k(s)\rangle ^{-\frac{1}{1-k}}    \big(\log \langle A_k(s) \rangle\big)^{-(p-1)} \big(\mathcal{U} (s)\big)^p\, \mathrm{d}s \\
&  \gtrsim    \langle A_k(t)\rangle ^{-1}  \int_1^t  \frac{\phi_k(t)-\phi_k(s)}{s}  \big(\log \langle A_k(s) \rangle\big)^{-(p-1)} \big(\mathcal{U} (s)\big)^p\, \mathrm{d}s,
\end{align*} which is precisely \eqref{iteration frame}. This completes the proof.
\end{proof}

\begin{lemma} \label{Lemma lower bound int |u|^p} Suppose that the assumptions in Corollary \ref{Corollary fund ineq} are satisfied. Then, there exists a positive constant $K=K(u_0,u_1,n,p,R,k)$ such that the lower bound estimate
\begin{align} \label{lower bound int |u|^p}
\int_{\mathbb{R}^n} |u(t,x)|^p \, \mathrm{d}x \geqslant K \varepsilon^p \langle A_k(t)\rangle^{(n-1)(1-\frac{p}{2})+\frac{kp}{2(1-k)}}
\end{align} holds for any $t\in (1,T)$. 
\end{lemma}

\begin{proof} We adapt the proof of Lemma 5.1 in \cite{WakYor18} to our model. Let us fix $q>(n-3)/2 +1/p'$. Combining \eqref{def functional mathcalU}, \eqref{mathcalU > epsilon} and H\"older's inequality, it results
\begin{align*}
\varepsilon t^{\frac{k}{2}} \lesssim t^{\frac{k}{2}} \mathcal{U}(t) =  \int_{\mathbb{R}^n} u(t,x) \, \xi_q(t,t,x;k) \, \mathrm{d} x \leqslant \left(\int_{\mathbb{R}^n}|u(t,x)|^p\, \mathrm{d}x\right)^{1/p} \left(\int_{B_{R+A_k(t)}}\big(\xi_q(t,t,x;k\big)^{p'}\, \mathrm{d}x\right)^{1/p'}.
\end{align*} Hence,
\begin{align} \label{lower bound int |u|^p intermediate}
\int_{\mathbb{R}^n}|u(t,x)|^p\, \mathrm{d}x \gtrsim \varepsilon^p t^{\frac{kp}{2}} \left(\int_{B_{R+A_k(t)}}\big(\xi_q(t,t,x;k\big)^{p'}\, \mathrm{d}x\right)^{-(p-1)}.
\end{align} Let us determine an upper bound estimates for the integral of $\xi_q(t,t,x;k)^{p'}$. By using \eqref{upper bound xi q t=s}, we have
\begin{align*}
\int_{B_{R+A_k(t)}}\big(\xi_q(t,t,x;k\big)^{p'}\, \mathrm{d}x & \lesssim \langle A_k(t)\rangle ^{-\frac{n-1}{2}p'} \int_{B_{R+A_k(t)}} \langle A_k(t) - |x|\rangle ^{(n-3)p'/2-p'q} \, \mathrm{d}x \\
& \lesssim \langle A_k(t)\rangle ^{-\frac{n-1}{2}p'} \int_0^{R+A_k(t)} r^{n-1} \langle A_k(t) - r\rangle ^{(n-3)p'/2-p'q} \, \mathrm{d}r \\
& \lesssim \langle A_k(t)\rangle ^{-\frac{n-1}{2}p'+n-1} \int_0^{R+A_k(t)}  \langle A_k(t) - r\rangle ^{(n-3)p'/2-p'q} \, \mathrm{d}r.
\end{align*} Performing the change of variable $A_k(t)-r=\varrho$, one gets
\begin{align*}
\int_{B_{R+A_k(t)}}\big(\xi_q(t,t,x;k\big)^{p'}\, \mathrm{d}x & \lesssim  \langle A_k(t)\rangle ^{-\frac{n-1}{2}p'+n-1} \int^{A_k(t)}_{-R}  (3+|\varrho|)^{(n-3)p'/2-p'q} \, \mathrm{d}\varrho \\
& \lesssim  \langle A_k(t)\rangle ^{-\frac{n-1}{2}p'+n-1}
\end{align*} because of $(n-3)p'/2 -p'q<-1$.

 Combining this upper bound estimates for the integral of $\xi_q(t,t,x;k)^{p'}$, \eqref{lower bound int |u|^p intermediate} and using $t\approx\langle A_k(t)\rangle^{\frac{1}{1-k}}$ for $t\geqslant 1$, we arrive at \eqref{lower bound int |u|^p}. The proof is over.
\end{proof}

In Proposition \ref{Proposition iteration frame}, we derive the iteration frame for $\mathcal{U}$. In the next result, we shall prove a first lower bound estimate for $\mathcal{U}$, which shall be the base case of the inductive argument in Section \ref{Subsection iteration procedure}.

\begin{proposition} \label{Proposition first logarithmic lower bound mathcalU}  Suppose that the assumptions in Corollary \ref{Corollary fund ineq} are satisfied and let $q=(n-1)/2-1/p$. Let $\mathcal{U}$ be defined by \eqref{def functional mathcalU}. Then, for $t\geqslant 3/2$ the functional $\mathcal{U}(t)$ fulfills
\begin{align} \label{first logarithmic lower bound mathcalU}
\mathcal{U}(t) \geqslant M \varepsilon^p \log\left(\tfrac{2t}{3}\right),
\end{align} where  the positive constant $M$ depends on $u_0,u_1,n,p,R,k$.
\end{proposition}

\begin{proof}
From \eqref{fundamental inequality functional mathcalU} we know that
\begin{align*}
\mathcal{U}(t) & \geqslant  t^{-\frac{k}{2}} \int_1^t  (\phi_k(t)-\phi_k(s)) \, s^{1-p}  \int_{\mathbb{R}^n} |u(s,x)|^p \eta_q(t,s,x;k) \, \mathrm{d}x \, \mathrm{d}s.
\end{align*}  Consequently, applying \eqref{lower bound eta q} first and then \eqref{lower bound int |u|^p}, we find
\begin{align*}
\mathcal{U}(t) & \geqslant B_1  \langle A_k(t)\rangle ^{-1} \int_1^t  (\phi_k(t)-\phi_k(s)) \, s^{1-p+\frac{k}{2}} \langle A_k(s)\rangle ^{-q} \int_{\mathbb{R}^n} |u(s,x)|^p  \, \mathrm{d}x \, \mathrm{d}s \\
& \geqslant B_1 K \varepsilon^p \langle A_k(t)\rangle ^{-1} \int_1^t  (\phi_k(t)-\phi_k(s)) \, s^{1-p+\frac{k}{2}}   \langle A_k(s)\rangle^{-q+(n-1)(1-\frac{p}{2})+\frac{kp}{2(1-k)}} \,  \mathrm{d}s \\
& \gtrsim \varepsilon^p \langle A_k(t)\rangle ^{-1} \int_1^t  (\phi_k(t)-\phi_k(s))    \langle A_k(s)\rangle^{(1-p+\frac{k}{2})\frac{1}{1-k}-\frac{n-1}{2}+\frac{1}{p}+(n-1)(1-\frac{p}{2})+\frac{kp}{2(1-k)}}  \, \mathrm{d}s \\
& \gtrsim  \varepsilon^p \langle A_k(t)\rangle ^{-1} \int_1^t  (\phi_k(t)-\phi_k(s))  \langle A_k(s)\rangle ^{-\left(\frac{n-1}{2}+\frac{2-k}{2(1-k)}\right)p+\left(\frac{n-1}{2}+\frac{2+k}{2(1-k)}\right)+\frac 1p}    \, \mathrm{d}s \\
& \gtrsim  \varepsilon^p \langle A_k(t)\rangle ^{-1} \int_1^t  (\phi_k(t)-\phi_k(s))  \langle A_k(s)\rangle ^{-\frac{1}{1-k}}    \, \mathrm{d}s  \gtrsim  \varepsilon^p \langle A_k(t)\rangle ^{-1} \int_1^t  \frac{\phi_k(t)-\phi_k(s)}{s} \, \mathrm{d}s. 
\end{align*} We estimate now the integral  in the right -- hand side of the previous chain of inequalities. Integration by parts leads to 
\begin{align*}
 \int_1^t  \frac{\phi_k(t)-\phi_k(s)}{s} \, \mathrm{d}s &  =  \big(\phi_k(t)-\phi_k(s)\big)\log s \, \Big|^{s=t}_{s=1} +\int_1^t  \phi_k'(s) \log s \, \mathrm{d}s \\ &= \int_1^t s^{-k} \log s \, \mathrm{d}s \geqslant t^{-k} \int_1^t  \log s \, \mathrm{d}s.
\end{align*} Therefore, for $t\geqslant 3/2$ 
\begin{align*}
\mathcal{U}(t) &   \gtrsim  \varepsilon^p \langle A_k(t)\rangle ^{-1} t^{-k} \int_1^t  \log s \, \mathrm{d}s \geqslant \varepsilon^p \langle A_k(t)\rangle ^{-1} t^{-k} \int_{2t/3}^t  \log s \, \mathrm{d}s \geqslant (1/3) \varepsilon^p \langle A_k(t)\rangle ^{-1} t^{1-k}   \log (2t/3) \\
& \gtrsim  \varepsilon^p  \log (2t/3) ,
\end{align*} where in the last line we employed $t\approx \langle A_k(t)\rangle ^{\frac{1}{1-k}}$ for $t\geqslant 1$. Also, the proof is complete.
\end{proof}

\subsection{Iteration argument} \label{Subsection iteration procedure}

In this section we prove the blow -- up result. More specifically, we are going to show a sequence of lower bound estimates for $\mathcal{U}$ and from these lower bound estimates we conclude that for $t$ over a certain $\varepsilon$~--~dependent threshold the functional $\mathcal{U}(t)$ may not be finite. 

Our goal is to show the validity of the sequence of lower bound estimates
\begin{align} \label{lower bound mathcalU j step}
\mathcal{U}(t)\geqslant C_j \big(\log\langle A_k(t)\rangle \big)^{-\beta_j} \left(\log \left(\frac{t}{\ell_j}\right)\right)^{\alpha_j} \qquad \mbox{for} \ t\geqslant \ell_j
\end{align}  for any $j\in \mathbb{N}$, where the bounded sequence of parameters characterizing the slicing procedure is $\{\ell_j\}_{j\in\mathbb{N}}$ with $\ell_j\doteq 2-2^{-(j+1)}$ and $\{C_j\}_{j\in\mathbb{N}},\{\alpha_j\}_{j\in\mathbb{N}},\{\beta_j\}_{j\in\mathbb{N}}$ are sequences of positive numbers that we will determine throughout the iteration argument.

In order to show \eqref{lower bound mathcalU j step}, we apply an inductive argument. As we have already said, the crucial idea here is to apply a slicing procedure for the domain of integration in the iteration frame \eqref{iteration frame}, in order to increase the power of the second logarithmic term in \eqref{lower bound mathcalU j step} step by step. This idea was introduced for the first time in \cite{AKT00} and since then it has been applied successfully to study the blow -- up dynamic of semilinear wave models in critical cases, overcoming the difficulties in the application of Kato's lemma for critical cases.

Since \eqref{lower bound mathcalU j step} is true in the base case $j=0$, provided that $C_0\doteq M\varepsilon^p$ and $\alpha_0\doteq 1, \beta_0=0$ (cf. Proposition \ref{Proposition first logarithmic lower bound mathcalU}), it remains to prove the inductive step. We assume \eqref{lower bound mathcalU j step} true for $j\geqslant 0$ and we have to prove it for $j+1$. Plugging \eqref{lower bound mathcalU j step} in \eqref{iteration frame}, for $t\geqslant \ell_{j+1}$ we get
\begin{align*}
\mathcal{U}(t) & \geqslant C C_j^p \langle A_k(t)\rangle^{-1}\int_{\ell_j}^t \frac{\phi_k(t)-\phi_k(s)}{s} \big(\log \langle A_k(s)\rangle\big)^{-(p-1) -\beta_jp} \left(\log \left(\tfrac{s}{\ell_j}\right)\right)^{\alpha_j p} \, \mathrm{d}s \\
& \geqslant C C_j^p \langle A_k(t)\rangle^{-1} \big(\log \langle A_k(t)\rangle\big)^{-(p-1) -\beta_jp}\int_{\ell_j}^t \frac{\phi_k(t)-\phi_k(s)}{s}  \left(\log \left(\tfrac{s}{\ell_j}\right)\right)^{\alpha_j p} \, \mathrm{d}s.
\end{align*} Using integration by parts, we find
\begin{align*}
& \int_{\ell_j}^t \frac{\phi_k(t)-\phi_k(s)}{s}  \left(\log \left(\tfrac{s}{\ell_j}\right)\right)^{\alpha_j p} \, \mathrm{d}s & \\ & \quad =  \big(\phi_k(t)-\phi_k(s)\big)   (\alpha_jp+1)^{-1}\left(\log \left(\tfrac{s}{\ell_j}\right)\right)^{\alpha_j p+1} \, \Big|^{s=t}_{s=\ell_j} +  (\alpha_j p+1)^{-1}  \int_{\ell_j}^t \phi'_k(s) \left(\log \left(\tfrac{s}{\ell_j}\right)\right)^{\alpha_j p+1}  \, \mathrm{d}s \\
& \quad =  (\alpha_j p+1)^{-1}  \int_{\ell_j}^t s^{-k} \left(\log \left(\tfrac{s}{\ell_j}\right)\right)^{\alpha_j p+1}  \, \mathrm{d}s  \geqslant  (\alpha_j p+1)^{-1} t^{-k} \int_{\ell_j}^t  \left(\log \left(\tfrac{s}{\ell_j}\right)\right)^{\alpha_j p+1}  \, \mathrm{d}s  \\
& \quad \geqslant  (\alpha_j p+1)^{-1} t^{-k} \int_{\tfrac{\ell_j t}{\ell_{j+1}}}^t  \left(\log \left(\tfrac{s}{\ell_j}\right)\right)^{\alpha_j p+1}  \, \mathrm{d}s \geqslant  (\alpha_j p+1)^{-1} t^{1-k} \left(1-\tfrac{\ell_j}{\ell_{j+1}}\right)   \left(\log \left(\tfrac{t}{\ell_{j+1}}\right)\right)^{\alpha_j p+1}   \\
& \quad \geqslant  (\alpha_j p+1)^{-1} 2^{-(j+3)} \gamma_k \langle A_k(t)\rangle   \left(\log \left(\tfrac{t}{\ell_{j+1}}\right)\right)^{\alpha_j p+1}  ,
\end{align*} where in the last step we applied $1-\ell_j/\ell_{j+1}>2^{-(j+3)}$ and $t^{1-k}\geqslant \gamma_k \langle A_k(t)\rangle $ for $t\geqslant 1$ with 
\begin{align*}
\gamma_k \doteq \begin{cases}  1/3 & \mbox{if} \ \, k\in[0,2/3], \\ (1-k) & \mbox{if} \ \, k\in[2/3,1).\end{cases}
\end{align*}
Therefore,
\begin{align*}
\mathcal{U}(t) & \geqslant C \gamma_k \, 2^{-(j+3)}  (\alpha_j p+1)^{-1}  C_j^p \big(\log \langle A_k(t)\rangle\big)^{-(p-1) -\beta_jp}   \left(\log \left(\tfrac{t}{\ell_{j+1}}\right)\right)^{\alpha_j p+1}
\end{align*} for $t\geqslant \ell_{j+1}$, that is, we proved \eqref{lower bound mathcalU j step} for $j+1$, provided that
\begin{align*}
C_{j+1} \doteq  C \gamma_k \, 2^{-(j+3)}  (\alpha_j p+1)^{-1}  C_j^p, \quad \alpha_{j+1}\doteq 1+p \alpha_j, \quad  \beta_{j+1}\doteq p-1 +p \beta_j.
\end{align*}

Next we establish a lower bound estimate for $C_j$. For this purpose, we provide first an explicit representation of the exponents $\alpha_j$ and $\beta_j$. Employing recursively the relations $\alpha_j=1+p\alpha_{j-1}$ and $\beta_j= (p-1) + p \beta_{j-1}$ and the initial exponents $\alpha_0=1$, $\beta_0=0$, we obtain 
\begin{align}\label{explit expressions aj and bj}
\alpha_j & = \alpha_0 p^j +\sum_{k=0}^{j-1} p^k =  \tfrac{p^{j+1}-1}{p-1} \quad \mbox{and} \quad
\beta_j = p^j \beta_0 +(p-1) \sum_{k=0}^{j-1} p^k = p^j-1.
\end{align} In particular, $\alpha_{j-1}p+1= \alpha_j\leqslant p^{j+1}/(p-1) $ implies that 
\begin{align} \label{lower bound Mj no.1}
C_j \geqslant D\, (2 p)^{-j} C^p_{j-1}
\end{align} for any $j\geqslant 1$, where $D\doteq {2^{-2}} C \gamma_k (p-1)/p$. Applying the logarithmic function to both sides of  \eqref{lower bound Mj no.1} and using iteratively the resulting inequality, we find
\begin{align*}
\log C_j & \geqslant p \log C_{j-1} -j \log (2p)+\log D \\
& \geqslant \ldots \geqslant p^j \log C_0 -\Bigg(\sum_{k=0}^{j-1}(j-k)p^k \Bigg)\log (2p)+\Bigg(\sum_{k=0}^{j-1} p^k \Bigg)\log D  \\
& = p^j \left(\log M \varepsilon^p -\frac{p\log (2p)}{(p-1)^2}+\frac{\log D }{p-1}\right)+\left( \frac{j}{p-1}+\frac{p}{(p-1)^2}\right)\log (2p)-\frac{\log D}{p-1},
\end{align*} where we used the identity 
\begin{align} \label{identity sum (j-k)p^k}
\sum\limits_{k=0}^{j-1}(j-k)p^k = \frac{1}{p-1}\left(\frac{p^{j+1}-p}{p-1}-j\right).
\end{align} Let us define $j_0=j_0(n,p,k)$ as the smallest nonnegative integer such that $$j_0\geqslant \frac{\log D}{\log (2p)}-\frac{p}{p-1}.$$ Hence, for any $j\geqslant j_0$ we have the estimate 
\begin{align} \label{lower bound Mj no.2}
\log C_j & \geqslant  p^j \left(\log M \varepsilon^p -\frac{p\log (2p)}{(p-1)^2}+\frac{\log D}{p-1}\right) = p^j \log (E \varepsilon^p),
\end{align} where $E\doteq M (2p)^{-p/(p-1)^2}D^{1/(p-1)}$. 

Combining \eqref{lower bound mathcalU j step}, \eqref{explit expressions aj and bj} and \eqref{lower bound Mj no.2}, we arrive at
\begin{align*}
\mathcal{U}(t)&\geqslant \exp \left( p^j\log(E\varepsilon^p)\right) \left(\log\langle A_k(t)\rangle \right)^{-\beta_j} \left(\log \left(\tfrac t2\right)\right)^{\alpha_j} \\
&= \exp \left( p^j\log(E\varepsilon^p)\right) \left(\log\langle A_k(t)\rangle \right)^{-p^j+1} \left(\log \left(\tfrac t2\right)\right)^{(p^{j+1}-1)/(p-1)} 
\\
&= \exp \left( p^j\log\left(E\varepsilon^p \left(\log\langle A_k(t)\rangle\right)^{-1}\left(\log \left(\tfrac t2\right)\right)^{p/(p-1)}\right) \right) \log\langle A_k(t)\rangle  \left(\log \left(\tfrac t2\right)\right)^{-1/(p-1)}
\end{align*} for $t\geqslant 2$ and any $j\geqslant j_0$.
Since for $t\geqslant t_0(k) \doteq \max\big\{4,\gamma_k^{-1/k}\big\}$ the inequalities $$\log\langle A_k(t)\rangle \leqslant (1-k) \log t-\log \gamma_k \leqslant \log t \quad \mbox{and} \quad \log (\tfrac t2)\geqslant 2^{-1} \log t $$ hold true, then,
\begin{align}
\mathcal{U}(t)&\geqslant  \exp \left( p^j\log\left(2^{-p/(p-1)}E\varepsilon^p \left(\log t\right)^{1/(p-1)}\right) \right) \log\langle A_k(t)\rangle  \left(\log \left(\tfrac t2\right)\right)^{-1/(p-1)} \label{final lower bound G}
\end{align} for $t\geqslant t_0$ and any $j\geqslant j_0$. Let us denote $J(t,\varepsilon)\doteq 2^{-p/(p-1)}E\varepsilon^p \left(\log t\right)^{1/(p-1)}$.

 If we choose $\varepsilon_0=\varepsilon_0(n,p,k,\lambda_0,R,u_0,u_1)$ sufficiently small so that
 \begin{align*}
 \exp \left(2^{p}E^{1-p}\varepsilon_0^{-p(p-1)}\right)\geqslant t_0,
 \end{align*}
 then, for any $\varepsilon\in (0,\varepsilon_0]$ and for $t> \exp \left(2^{p}E^{1-p}\varepsilon^{-p(p-1)}\right)$ we get $t\geqslant t_0$ and $J(t,\varepsilon)>1$. Consequently, for any $\varepsilon\in (0,\varepsilon_0]$ and for $t> \exp \left(2^{p}E^{1-p}\varepsilon^{-p(p-1)}\right)$ letting $j\to \infty$ in \eqref{final lower bound G} it results that the lower bound for $\mathcal{U}(t)$ blows up; hence, $\mathcal{U}(t)$ is not finite as well. Also, we showed  that $\mathcal{U}$ blows up in finite time and, moreover, we proved the upper bound estimate for the lifespan $$T(\varepsilon)\leqslant \exp \left(2^{p}E^{1-p}\varepsilon^{-p(p-1)}\right).$$ 
Therefore, we completed the proof of Theorem \ref{Theorem critical case p0}.

\section{Semilinear wave equation in EdeS spacetime: subcritical case} \label{Section subcritical case}

As byproduct of the approach developed in Section \ref{Section critical case p0}, we derive in this section the upper bound estimates for the lifespan of local in time solutions in the subcritical case $1<p<\max\{p_0(n,k),p_1(n,k)\}$. Our main tool will be the generalization of Kato's lemma containing the upper bound estimates for the lifespan proved in \cite{Tak15}, whose statement is recalled below for the ease of the reader.

\begin{lemma} \label{Kato's lemma} Let $p>1$, $a>0$, $q>0$ satisfy $$M\doteq \frac{p-1}{2}a-\frac{q}{2}+1>0.$$ Assume that $F\in \mathcal{C}^2([\tau,T))$ satisfies
\begin{align}
 & F(t)  \geqslant A t^a  \  \  \, \qquad \qquad  \qquad \qquad \mbox{for} \ \ t\geqslant T_0\geqslant \tau , \label{F lower bound Kato lemma} \\
 & F''(t)  \geqslant B (t+R)^{-q}|F(t)|^{p} \qquad \mbox{for} \ \ t\geqslant \tau, \label{F'' lower bound Kato lemma} \\
&  F(\tau)  \geqslant 0,  \ \ F'(\tau)>0, \label{F(0), F'(0) conditions Kato lemma}
\end{align} where $A,B,R,T_0$ are positive constants. Then, there exists a positive constant $C_0=C_0(p,a,q,B,\tau)$ such that 
\begin{align} 
T< 2^{\frac{2}{M}}T_1 \label{upper bound T Kato lemma}
\end{align} 
 holds, provided that
\begin{align}
T_1\doteq \max\left\{T_0,\frac{F(\tau)}{F'(\tau)},R\right\} \geqslant C_0 A^{-\frac{p-1}{2M}}. \label{lower bound T1 Kato lemma}
\end{align}
\end{lemma}

As we are going to apply this generalization of Kato's lemma, we will find some estimates already obtained in \cite{GalYag17EdS} in the treatment of the subcritical case, although the proofs that lead to these estimates are different.

Let us assume that $u_0,u_1$ are compactly supported with supports in $B_R$ for some $R>0$, nonnegative and nontrivial functions. Let $u$ be a solution on $[1,T)$ of \eqref{Semi EdeS k} according to Definition \ref{Def energy sol} such that $$\supp u(t,\cdot) \subset B_{R+A_k(t)}$$ for any $t\in (1,T)$, where $T=T(\varepsilon)$ is the lifespan of $u$.

Hence, we introduce as time -- dependent functional the spatial average of $u$
\begin{align} \label{def U}
\mathrm{U}(t) \doteq \int_{\mathbb{R}^n} u(t,x) \, \mathrm{d}x.
\end{align} Choosing a test function $\psi$ such that $\psi = 1$ on $\{(s,x)\in [1,t]\times \mathbb{R}^n: |x|\leqslant R+A_k(s) \}$ in \eqref{integral identity def energy sol}, we get
\begin{align*}
\mathrm{U}'(t) = \mathrm{U}'(1) +\int_1^t \int_{\mathbb{R}^n} s^{1-p} |u(s,x)|^p \, \mathrm{d}x \, \mathrm{d}s.
\end{align*} Also, differentiating the previous identity with respect to $t$, it results
\begin{align} \label{equality U''}
\mathrm{U}''(t) = t^{1-p}\int_{\mathbb{R}^n} |u(t,x)|^p \, \mathrm{d}x.
\end{align} By using the support condition for $u$ and H\"older's inequality, from the above inequality we obtain
\begin{align}
\mathrm{U}''(t) & \gtrsim t^{1-p} (R+A_k(t))^{-n(p-1)} |\mathrm{U}(t)|^p  \notag \\ &\gtrsim  (R+t)^{-((1-k)n+1)(p-1)} |\mathrm{U}(t)|^p \label{differential ineq U''} 
\end{align} for any $t\in(1,T)$.

Let us derive now two estimates from below for $\mathrm{U}$.
On the one hand, thanks to the convexity of $\mathrm{U}$, we have immediately
\begin{align} \label{U > epsilon}
\mathrm{U}(t)\geqslant \mathrm{U}(1)+ (t-1) \mathrm{U}'(1) \gtrsim \varepsilon t
\end{align} for any $t\in (1,T)$, where we used that $u_0, u_1$ are nonnegative and nontrivial in the unexpressed multiplicative constant. Plugging \eqref{U > epsilon} in \eqref{differential ineq U''} and integrating twice, we get
\begin{align}\label{lower bound U convex}
\mathrm{U}(t)\gtrsim \varepsilon^p t^{-((1-k)n+1)(p-1) +p+2}
\end{align}
for any $t\in [T_0,T)$, where $T_0>1$. The first lower bound estimate for $\mathrm{U}$ in \eqref{lower bound U convex} has been obtained from the convexity of  $\mathrm{U}$. On the other hand, from Lemma \ref{Lemma lower bound int |u|^p} and \eqref{equality U''}, integrating twice, we find a second lower bound estimate for $U$, that is,
\begin{align} \label{lower bound U int |u|^p}
\mathrm{U}(t)\gtrsim \varepsilon^p t^{(1-k)(n-1)(1-\frac{p}{2})+\frac{kp}{2} +1-p+2}
\end{align} for any $t\in [T_0,T)$.

Next we apply Lemma \ref{Kato's lemma} to the functional $\mathrm{U}$. Since $u_0,u_1$ are nonnegative and nontrivial we have $\mathrm{U}(1),\mathrm{U}'(1)>0$, so \eqref{F(0), F'(0) conditions Kato lemma} is fulfilled. Moreover, \eqref{differential ineq U''} corresponds to \eqref{F'' lower bound Kato lemma} with $q\doteq((1-k)n+1)(p-1)$. Finally, combining \eqref{lower bound U convex} and \eqref{lower bound U int |u|^p} we have \eqref{F lower bound Kato lemma} with $a= \max\{a_1,a_2\}$,  where 
\begin{align*}
a_1 &\doteq -((1-k)n+1)(p-1) +p+2, \\ a_2 & \doteq  (1-k)(n-1)(1-\tfrac{p}{2})+\tfrac{kp}{2} +1-p+2
\end{align*}  and $A\approx \varepsilon^p$.
According to this choice we have two possible value for the quantity $M$ in Lemma \ref{Kato's lemma}: either we use \eqref{lower bound U convex}, that is, $a=a_1$ and, consequently,
\begin{align*}
M_1\doteq \tfrac{p-1}{2} a_1 - \tfrac q2 +1 = \tfrac p2 \left [ -(1-k)n (p-1) +2\right]
\end{align*} or we use \eqref{lower bound U int |u|^p}, that is, $a=a_2$ and, then,
\begin{align*}
M_2\doteq \tfrac{p-1}{2} a_2 - \tfrac q2 +1 = \tfrac 12 \left \{ -\left[(1-k) \tfrac{n-1}{2} +1-\tfrac{k}{2}\right] p^2+\left[ (1-k) \tfrac{n+1}{2}+1+\tfrac{3k}{2} \right]p+1-k\right\}.
\end{align*} Therefore, for $M\doteq \max\{M_1,M_2\}>0$ Lemma \ref{Kato's lemma} provides a blow -- up result and the upper bound estimate for the lifespan $$T\lesssim \varepsilon ^{-\frac{p(p-1)}{2M}}.$$ Let us make the condition $M>0$ more explicit. The condition $M_1>0$ is equivalent to $p<p_1(n,k)$, while the condition $M_2>0$ is equivalent to $p<p_0(n,k)$. Hence, Lemma \ref{Kato's lemma} implies the validity of a blow -- up result for \eqref{Semi EdeS k} in the subcritical case $1<p< \max\{p_0(n,k),p_1(n,k)\}$ (exactly as in \cite{GalYag17EdS}) and the upper bound estimates for  the lifespan
\begin{align} \label{lifespan estimate subcrit case}
T(\varepsilon) \lesssim  \begin{cases} \varepsilon ^{- \left(\frac{2}{p-1}-(1-k) n\right)^{-1}} & \mbox{if} \ p<p_1(n,k), \\ \varepsilon ^{- \frac{p(p-1)}{\theta(p,n,k)}} & \mbox{if} \ p<p_0(n,k),\end{cases}
\end{align} where 
\begin{align}\label{def theta}
\theta(p,n,k) \doteq 1-k +\left[ (1-k) \tfrac{n+1}{2}+1+\tfrac{3k}{2} \right]p  -\left[(1-k) \tfrac{n-1}{2} +1-\tfrac{k}{2}\right] p^2.
\end{align}

Furthermore, we point out that $a > 1$ (so, in particular, $a>0$ as it is required in the assumptions of Lemma \ref{Kato's lemma}) if and only if $1<p <\max\{p_1(n,k),p_2(n,k)\}$, where
\begin{align*}
p_2(n,k)\doteq 2+\frac{2k}{(1-k)n+1}.
\end{align*} We want to show now that the condition $a> 1$ is always fulfilled whenever $M>0$ holds.
For this purpose, we shall determine how to order the exponents  $p_0,p_1,p_2$. Since $p_0(n,k)$ is defined through \eqref{intro equation critical exponent general case}, the inequality $p_0(n,k)>p_1(n,k)$ holds if and only if 
\begin{align*}
& \left((1-k)n +1\right)p_1(n,k)^2- \left((1-k)n +3+2k\right)p_1(n,k) -2(1-k)<0. 
\end{align*} By straightforward computations it follows that the last inequality is fulfilled if and only if $ n> N(k)$, where $N(k)$ is defined in \eqref{def N(k)}. Similarly, $p_0(n,k)>p_2(n,k)$  if and only if $n< N(k)$.
Summarizing,
\begin{equation}\label{order p0,p1,p2}
\begin{split}
p_2(n,k)<p_0(n,k)<p_1(n,k) & \qquad \mbox{if} \ \ n<N(k), \\
p_0(n,k)=p_1(n,k)=p_2(n,k) & \qquad \mbox{if} \ \ n=N(k), \\
p_1(n,k)<p_0(n,k)<p_2(n,k) & \qquad \mbox{if} \ \ n>N(k). 
\end{split}
\end{equation} Consequently, for $n\geqslant N(k)$ the critical condition is $p=p_0(n,k)$, so if $p< p_0(n,k)$, in particular, the condition $p< p_2(n,k)$ is fulfilled (that is, $M_2>0$ implies $a_2>1$). On the other hand, for $n < N(k)$ it holds $p_2(n,k)<p_1(n,k)$ and the condition $M_1>0$ and $a_1>1$ are both equivalent to $p<p_1(n,k)$ (the critical condition is $p=p_1(n,k)$ in this case). Therefore, we actually proved that $M>0$ implies $a>1$.
 
\begin{remark} In \cite{GalYag17EdS} the condition in the subcritical case on $p$ under which a blow -- up result holds for \eqref{Semi EdeS k} is written in a slightly different but equivalent way. Indeed, combining \cite[Equation (1.9)]{GalYag17EdS} with \eqref{order p0,p1,p2}, we see immediately that the condition for $p$ in \cite[Theorem 1.3]{GalYag17EdS} is satisfied if and only if $1<p<\max\{p_1(n,k),p_0(n,k)\}$.
\end{remark}

 Finally, we want to compare the upper bound estimates for the lifespan in \eqref{lifespan estimate subcrit case}. Clearly, the estimates
 \begin{align*}
T(\varepsilon) \lesssim  \begin{cases} \varepsilon ^{- \left(\frac{2}{p-1}-(1-k) n\right)^{-1}} & \mbox{if} \ n<N(k) \ \mbox{and} \ p\in [p_0(n,k),p_1(n,k)), \\ \varepsilon ^{- \frac{p(p-1)}{\theta(p,n,k)}} & \mbox{if}  \ n>N(k) \ \mbox{and} \ p\in [p_1(n,k),p_0(n,k)),\end{cases}
 \end{align*} cannot be improved because it holds either $p\geqslant p_0(n,k)$ or $p\geqslant p_1(n,k)$. Note that $p_2(n,k)$ plays no role in the determination of the upper bound estimate for the lifespan. 
 
  However, in the case $1<p<\min\{p_0(n,k),p_1(n,k)\}$ it is not clear which of the upper bounds in \eqref{lifespan estimate subcrit case} is better. Of course, in this case we have to compare $a_1$ and $a_2$. A straightforward computation shows that $a_1\geqslant a_2$ if and only if 
  \begin{align} \label{inequality p3}
  ((1-k)n-1)p\leqslant 2(1-k). 
  \end{align} If $n\leqslant \widetilde{N}(k)\doteq 1/(1-k)$, then, the previous inequality is always true. On the other hand, for $n>\widetilde{N}(k)$ we may introduce the further exponent $$p_3(n,k)\doteq \frac{2(1-k)}{(1-k)n-1}.$$ It turns out that $p_3(n,k)>1$ if an only if $\widetilde{N}(k) < n< \widehat{N}(k)\doteq 2+1/(1-k)$. Moreover, for $n>\widetilde{N}(k)$ the inequalities  $p_1(n,k)<p_3(n,k)$ and  $p_0(n,k)<p_3(n,k)$ are both satisfied if and only if $n< N(k)$. 
  
  In order to clarify the upper bound estimates in \eqref{lifespan estimate subcrit case}, we shall consider five different subcases depending on the range for the spatial dimension $n$.
  
\subsubsection*{Case $n\leqslant \widetilde{N}(k)$ }

In this case, \eqref{inequality p3} is always satisfied as the left -- hand side is nonpositive. So, $a_1\geqslant a_2$. Therefore, for any $1<p<p_1(n,k)$ the following upper bound estimate holds
\begin{align} \label{lifespan estimate subcrit a1}
T(\varepsilon) \lesssim \varepsilon ^{- \left(\frac{2}{p-1}-(1-k) n\right)^{-1}}.
\end{align} 

\subsubsection*{Case $\widetilde{N}(k) < n <N(k)$ } In this case, \eqref{inequality p3} is satisfied for $p\leqslant p_3$. Hence, by the ordering $1<p_0(n,k) <p_1(n,k) <p_3(n,k)$, we get that $a_1>a_2$ for exponents satisfying $1<p<p_1(n,k)$. Also, even in this case \eqref{lifespan estimate subcrit a1} is a better estimates than $T(\varepsilon) \lesssim \varepsilon ^{- \frac{p(p-1)}{\theta(p,n,k)}}$. 

\subsubsection*{Case $n = N(k)$ } In this limit case, $p_0(n,k) =p_1(n,k) =p_3(n,k)$. So, for $1<p<p_1(n,k)=p_3(n,k)$ it holds $a_1>a_2$ and as in the previous case \eqref{lifespan estimate subcrit a1} is the best estimate.

\subsubsection*{Case $N(k) < n < \widehat{N}(k)$ } In this case, it results $1<p_3(n,k) <p_1(n,k)<p_0(n,k) $. So, for $1<p\leqslant p_3(n,k)$ it holds $a=a_1$, while for $p_3(n,k) <p<p_0(n,k)$ we have $a=a_2$. Therefore,
\begin{align*}
T(\varepsilon) \lesssim  \begin{cases} \varepsilon ^{- \left(\frac{2}{p-1}-(1-k) n\right)^{-1}} & \mbox{if} \ \ p\in (1, p_3(n,k)], \\ \varepsilon ^{- \frac{p(p-1)}{\theta(p,n,k)}} & \mbox{if}  \ \ p\in (p_3(n,k),p_0(n,k)).\end{cases}
\end{align*}

\subsubsection*{Case $ n \geqslant \widehat{N}(k)$ } In this case, $p_3(n,k)\leqslant 1$ and $1<p_1(n,k)<p_0(n,k)$ so \eqref{inequality p3} is never satisfied for $p>1$. Hence, $a_2>a_1$ for any $1<p<p_0(n,k)$, that is,
\begin{align*}
T(\varepsilon) \lesssim \varepsilon ^{- \frac{p(p-1)}{\theta(p,n,k)}}
\end{align*} is a better estimate than \eqref{lifespan estimate subcrit a1}.

\subsection{Lifespan estimates in the subcritical case}

Summarizing, what we established in the above subcases, we proved the following proposition, that completes \cite[Theorem 1.3]{GalYag17EdS} with   the estimate for the lifespan while Theorem \ref{Theorem critical case p0} and Theorem \ref{Theorem critical case p1} deal with the critical case that was not discussed in  \cite{GalYag17EdS}.

\begin{proposition}  \label{Proposition lifespan subcrit}
Let $n\geqslant 1$ and $1<p< \max\{p_0(n,k), p_1(n,k)\}$. Let us assume that $u_0\in H^1(\mathbb{R}^n)$ and $u_1\in L^2(\mathbb{R}^n)$ are nonnegative, nontrivial and compactly supported functions with supports contained in $B_R$ for some $R>0$. Let $$u\in \mathcal{C} \big([1,T), H^1(\mathbb{R}^n)\big) \cap \mathcal{C}^1 \big([1,T), L^2(\mathbb{R}^n)\big)\cap L^p_{\mathrm{loc}}\big([1,T)\times \mathbb{R}^n\big)$$ be an energy solution to \eqref{Semi EdeS k} according to Definition \ref{Def energy sol}  with lifespan $T=T(\varepsilon)$ and fulfilling the support condition $\mathrm{supp} \, u(t,\cdot)\subset B_{A_k(t)+R}$ for any $t\in (1,T)$. Then, there exists a positive constant $\varepsilon_0=\varepsilon_0(u_0,u_1,n,p,k,R)$ such that for any $\varepsilon\in (0,\varepsilon_0]$ the energy solution $u$ blows up in finite time. Furthermore, the upper bound estimates for the lifespan
\begin{align*}
T(\varepsilon)\leqslant \begin{cases} C \varepsilon ^{- \left(\frac{2}{p-1}-(1-k) n\right)^{-1}} & \mbox{if}  \ n\leqslant N(k) \ \mbox{and} \  p\in (1, p_1(n,k)), \\ 
C \varepsilon ^{- \left(\frac{2}{p-1}-(1-k) n\right)^{-1}} & \mbox{if}  \ n\in( N(k) ,\widehat{N}(k) ) \ \mbox{and} \   \ p\in (1, p_3(n,k)], \\ 
C \varepsilon ^{- \frac{p(p-1)}{\theta(p,n,k)}} & \mbox{if}  \ n\in( N(k) ,\widehat{N}(k) ) \ \mbox{and} \ \ p\in (p_3(n,k),p_0(n,k)), \\
C \varepsilon ^{- \frac{p(p-1)}{\theta(p,n,k)}} & \mbox{if}  \ n\geqslant \widehat{N}(k) \ \mbox{and} \ \ p\in (1,p_0(n,k)),\end{cases}
\end{align*} hold, where the constant $C>0$ is independent of $\varepsilon$ and $\theta(p,n,k)$ is defined by \eqref{def theta}.
\end{proposition}

\section[Semilinear wave equation in EdeS spacetime: critical case $p=p_1(n,k)$]{Semilinear wave equation in EdeS spacetime: 2nd critical case}  \label{Section critical case p1}

In Section \ref{Section subcritical case} we derived the upper bound for the lifespan in the subcritical case, while in Section \ref{Section critical case p0} we studied the critical case $p=p_0(n,k)$. We have already remarked that $p=p_0(n,k)$ is the critical case when $n >N(k)$. Therefore, it remains to consider the critical case $p=p_1(n,k)$ when $n\leqslant N(k)$. In this section, we are going to prove a blow -- up result even in this critical case $p=p_1(n,k)$ and to derive the corresponding upper bound estimate for the lifespan. Even in this critical case, our approach will be based on a basic iteration argument combined with the slicing procedure we already applied in Section \ref{Section critical case p0}.

As time -- depending functional we will use the same one employed in Section \ref{Section subcritical case}, namely the function $\mathrm{U}$ defined in \eqref{def U}. Then, since $p=p_1(n,k)$ is equivalent to the condition 
\begin{align}\label{condition p=p1 equiv}
((1-k)n+1)(p-1)=p+1,
\end{align}
 we may rewrite \eqref{differential ineq U''} as 
\begin{align}\label{iteration frame 2nd crit case}
\mathrm{U}(t) \geqslant C \int_1^t\int_1^s (R+\tau)^{-(p+1)}\big(\mathrm{U}(\tau)\big)^p \, \mathrm{d}\tau \, \mathrm{d}s
\end{align} for any $t\in (1,T)$ and  for a suitable positive constant $C>0$. Let us point out that \eqref{iteration frame 2nd crit case} will be the iteration frame in the iteration procedure for the critical case $p=p_1(n,k)$. 

We know that $\mathrm{U}(t)\geqslant K \varepsilon\, t$ for any $t\in (1,T)$, where $K$ is a suitable positive constant, provided that $u_0,u_1$ are nonnegative, nontrivial and compactly supported (cf. the estimate in \eqref{U > epsilon}). Therefore,
\begin{align}
\mathrm{U}(t) & \geqslant C K^p \varepsilon^p \int_1^t\int_1^s (R+\tau)^{-(p+1)}\tau ^p \, \mathrm{d}\tau \, \mathrm{d}s  \geqslant C K^p (R+1)^{-(p+1)} \varepsilon^p \int_1^t\int_1^s \tau ^{-1} \, \mathrm{d}\tau \, \mathrm{d}s \notag \\
& = C K^p (R+1)^{-(p+1)} \varepsilon^p \int_1^t\log s \, \mathrm{d}s  \geqslant C K^p (R+1)^{-(p+1)} \varepsilon^p \int_{2t/3}^t\log s \, \mathrm{d}s \notag \\ &\geqslant  3^{-1} C K^p (R+1)^{-(p+1)} \varepsilon^p \, t \log \left( \tfrac{2t}{3}\right) \label{1st lower bound U p=p1}
\end{align} for $t \geqslant \ell_0=3/2$, where we used $R+\tau\leqslant (R+1)\tau$ for $\tau\geqslant 1$.

Hence, by using recursively \eqref{iteration frame 2nd crit case}, we are going to prove now the sequence of lower bound estimates
\begin{align}\label{lower bound U j p=p1}
\mathrm{U}(t)\geqslant K_j \, t \left(\log \left(\frac{t}{\ell_j}\right)\right)^{\sigma_j} \qquad \mbox{for} \ t\geqslant \ell_j
\end{align} for any $j\in \mathbb{N}$, where the sequence of parameters $\{\ell_j\}_{j\in\mathbb{N}}$ is defined as in Section \ref{Subsection iteration frame}, i.e. $\ell_j=2-2^{-(j+1)}$, and $\{K_j\}_{j\in\mathbb{N}}$, $\{\sigma\}_{j\in\mathbb{N}}$ are sequences of positive reals that we shall determine afterwards.

We remark that for $j=0$ \eqref{lower bound U j p=p1} holds true thanks to \eqref{1st lower bound U p=p1}, provided that $K_0= (CK^p (R+1)^{-(p+1)} \varepsilon^p)/3$ and $\sigma_0=1$. Next we are going to prove \eqref{lower bound U j p=p1} by using an inductive argument. Assumed the validity of \eqref{lower bound U j p=p1} for some $j\geqslant 0$ we have to prove \eqref{lower bound U j p=p1} for $j+1$. For this purpose, we plug \eqref{lower bound U j p=p1} in \eqref{iteration frame 2nd crit case}, thus, after shrinking the domain of integration, we have
\begin{align*}
\mathrm{U}(t) & \geqslant CK_j^p \int_{\ell_j}^t\int_{\ell_j}^s (R+\tau)^{-(p+1)} \tau^p \left(\log \left( \tfrac{\tau}{\ell_j}\right)\right)^{\sigma_j p} \mathrm{d}\tau \, \mathrm{d}s \\
 & \geqslant C(R+1)^{-(p+1)} K_j^p \int_{\ell_j}^t\int_{\ell_j}^s \tau^{-1}  \left(\log \left( \tfrac{\tau}{\ell_j}\right)\right)^{\sigma_j p} \mathrm{d}\tau \, \mathrm{d}s \\ & = C(R+1)^{-(p+1)} K_j^p (\sigma_j p+1)^{-1}\int_{\ell_j}^t  \left(\log \left( \tfrac{s}{\ell_j}\right)\right)^{\sigma_j p+1}   \mathrm{d}s
\end{align*} for $t\geqslant \ell_{j+1}$. If we shrink the domain of integration to $[(\ell_j/\ell_{j+1})t,t]$ in the last $s$ -- integral we get
\begin{align*}
\mathrm{U}(t) & \geqslant  C (R+1)^{-(p+1)} K_j^p (\sigma_j p+1)^{-1}\int_{\tfrac{\ell_j t}{\ell_{j+1}}}^t  \left(\log \left( \tfrac{s}{\ell_j}\right)\right)^{\sigma_j p+1}   \mathrm{d}s \\ &\geqslant  C (R+1)^{-(p+1)} K_j^p (\sigma_j p+1)^{-1} \left(1- \tfrac{\ell_j}{\ell_{j+1}}\right) t \left(\log \left( \tfrac{t}{\ell_{j+1}}\right)\right)^{\sigma_j p+1}  \\
& \geqslant  C (R+1)^{-(p+1)} 2^{-(j+3)}K_j^p (\sigma_j p+1)^{-1}  t \left(\log \left( \tfrac{t}{\ell_{j+1}}\right)\right)^{\sigma_j p+1}  
\end{align*} for $t\geqslant \ell_{j+1}$, where in the last step we applied the inequality $1-\ell_j/\ell_{j+1}>2^{-(j+3)}$. Also, we proved \eqref{lower bound U j p=p1} for $j+1$ provided that 
\begin{align*}
K_{j+1}\doteq  C (R+1)^{-(p+1)} 2^{-(j+3)} (\sigma_j p+1)^{-1} K_j^p \quad \mbox{and} \quad \sigma_{j+1} \doteq \sigma_j p+1.
\end{align*}

Next we determine a lower bound estimate for $K_j$.  First we find the value of the exponent $\sigma_j$. Applying iteratively the relation $\sigma_j=1+p\sigma_{j-1}$ and the initial exponent $\sigma_0=1$, we get 
\begin{align}\label{explit expressions sigmaj}
\sigma_j & = \sigma_0 p^j +\sum_{k=0}^{j-1} p^k =  \tfrac{p^{j+1}-1}{p-1}.
\end{align} In particular, $\sigma_{j-1}p+1= \sigma_j\leqslant p^{j+1}/(p-1) $ implies that 
\begin{align} \label{lower bound Kj no.1}
K_j \geqslant L\, (2 p)^{-j} K^p_{j-1}
\end{align} for any $j\geqslant 1$, where $L\doteq {2^{-2}} C (R+1)^{-(p+1)} (p-1)/p$. Applying the logarithmic function to both sides of  \eqref{lower bound Kj no.1} and reusing the resulting inequality in an iterative way, we arrive at
\begin{align*}
\log K_j & \geqslant p \log K_{j-1} -j \log (2p)+\log L \\
& \geqslant \ldots \geqslant p^j \log K_0 -\Bigg(\sum_{k=0}^{j-1}(j-k)p^k \Bigg)\log (2p)+\Bigg(\sum_{k=0}^{j-1} p^k \Bigg)\log L  \\
& = p^j \left(\log \left(3^{-1}CK^p (R+1)^{-(p+1)} \varepsilon^p\right) -\frac{p\log (2p)}{(p-1)^2}+\frac{\log L }{p-1}\right)+\left( \frac{j}{p-1}+\frac{p}{(p-1)^2}\right)\log (2p)-\frac{\log L}{p-1},
\end{align*} where we applied again the identity \eqref{identity sum (j-k)p^k}. Let us define $j_1=j_1(n,p,k)$ as the smallest nonnegative integer such that $$j_1\geqslant \frac{\log L}{\log (2p)}-\frac{p}{p-1}.$$ Hence, for any $j\geqslant j_1$ the estimate 
\begin{align} \label{lower bound Kj no.2}
\log K_j & \geqslant  p^j \left(\log \left( 3^{-1} CK^p (R+1)^{-(p+1)}  \varepsilon^p\right) -\frac{p\log (2p)}{(p-1)^2}+\frac{\log L}{p-1}\right) = p^j \log ( N \varepsilon^p)
\end{align} holds, where $N\doteq 3^{-1}CK^p (R+1)^{-(p+1)} (2p)^{-p/(p-1)^2}L^{1/(p-1)}$. 

Combining \eqref{lower bound U j p=p1}, \eqref{explit expressions sigmaj} and \eqref{lower bound Kj no.2}, we arrive at
\begin{align*}
\mathrm{U}(t)&\geqslant \exp \left( p^j\log(N\varepsilon^p)\right) t \left(\log \left(\tfrac t{\ell_j}\right)\right)^{\sigma_j}  \\ & \geqslant \exp \left( p^j\log(N\varepsilon^p)\right) t \left( \tfrac 12 \log t \right)^{(p^{j+1}-1)/(p-1)} \\
&= \exp \left( p^j\log\left( 2^{-p/(p-1)}N\varepsilon^p \left(\log t \right)^{p/(p-1)}\right) \right) t \left( \tfrac 12 \log  t \right)^{-1/(p-1)} 
\end{align*} for $t\geqslant 4$ and for any $j\geqslant j_1$, where we applied the inequality $\log (t/ \ell_j)\geqslant \log(t/2) \geqslant (1/2) \log t$ for all $t\geqslant 4$.
 If we denote $H(t,\varepsilon)\doteq 2^{-p/(p-1)}N\varepsilon^p \left(\log t\right)^{p/(p-1)}$, the last estimate may be rewritten as
 \begin{align}\label{final lower bound U}
 \mathrm{U}(t)&\geqslant \exp \big( p^j \log H(t,\varepsilon)\big) t \left( \tfrac 12 \log t \right)^{-1/(p-1)} 
 \end{align} for $t\geqslant 4$ and any $j\geqslant j_1$.

 Let us fix  $\varepsilon_0=\varepsilon_0(n,p,k,R,u_0,u_1)$ so that
 \begin{align*}
 \exp \left(2N^{-(1-p)/p}\varepsilon_0^{-(p-1)}\right)\geqslant 4.
 \end{align*}
 Then, for any $\varepsilon\in (0,\varepsilon_0]$ and for $t> \exp \left(2N^{-(1-p)/p}\varepsilon^{-(p-1)}\right)$ we get $t\geqslant 4$ and $H(t,\varepsilon)>1$. Thus, for any $\varepsilon\in (0,\varepsilon_0]$ and for $t> \exp \left(2N^{-(1-p)/p}\varepsilon^{-(p-1)}\right)$ as $j\to \infty$ in \eqref{final lower bound U} we find that the lower bound for $\mathrm{U}(t)$ blows up and, consequently, $\mathrm{U}(t)$ cannot be finite too. Summarizing, we proved  that $\mathrm{U}$ blows up in finite time and, besides, we showed the upper bound estimate for the lifespan $$T(\varepsilon)\leqslant \exp \left(2N^{-(1-p)/p}\varepsilon^{-(p-1)}\right).$$ 

Altogether, we established Theorem \ref{Theorem critical case p1} in the critical case $p=p_1(n,k)$.

\begin{remark} Combining the results from Theorems \ref{Theorem critical case p0} and \ref{Theorem critical case p1} and Proposition \ref{Proposition lifespan subcrit}, we a full picture of the upper bound estimates for the lifespan of local in time solutions to \eqref{Semi EdeS k} whenever $1<p\leqslant \max\{p_0(n,k),p_1(n,k)\}$, of course, under suitable sign, size and support assumptions for the initial data.
\end{remark}


\section{Final remarks}

Let us compare our results with the corresponding ones for the semilinear wave equation in the flat case. First, we point out that due to the presence of the term $t^{1-p}$ in the semilinear term in \eqref{Semi EdeS k u}, we have a competition between the two exponents $p_0, p_1$ to be the critical exponent. This for the classical semilinear wave equation with power nonlinearity does not happen since $p_{\mathrm{Str}}(n)\geqslant \frac{n+1}{n-1}$ for any $n\geqslant 2$. However, a similar situation it has been observed when lower order terms with time -- dependent coefficients in the \emph{scale -- invariant case} are present, with a competition between a shift of Fujita exponent and a shift Strauss exponent (cf. \cite
{DLR15,DL15,NPR16,PalRei18,PT18,Pal18odd,Pal18even}).
On the other hand, the presence of the exponent $p_3$ for dimensions $n\in (N(k),\widehat{N}(k))$ to distinguish among two different upper bounds for the lifespan depending on the range for $p$ is exactly what happens for the semilinear wave equation in spatial dimensions $n=2$ (see \cite{Tak15,IKTW19}). Moreover, the situation for \eqref{Semi EdeS k} when $n\leqslant N(k)$ is completely analogous to what happens for the semilinear wave equation when $n=1$, see \cite{Zhou92} for the Euclidean case.

After the completion of the final version of this work, we found out the existence of the paper \cite{TW20}, where a more general model is considered. We point out that the approach we used in the critical case is completely different, and that we slightly improved their result in the special case of the semilinear wave equation in the generalized Einstein -- de Sitter spacetime, by removing the assumption on the size of the support of the Cauchy data (cf. \cite[Theorem 2.3]{TW20}).

\section*{Acknowledgments}

A. Palmieri 
is supported by the GNAMPA project `Problemi stazionari e di evoluzione nelle equazioni di campo nonlineari dispersive'. The author acknowledges Karen Yagdjian (UTRGV) and Hiroyuki Takamura (Tohoku Univ.) for valuable discussions on the model considered in this work.

\appendix

\section{Alternative proof of Proposition \ref{Proposition representations y0 and y1} in the special case $k=2/3$  } \label{Appendix solutions y''-lambda^2 t^(-4/3)y=0}

In this appendix we determine the representation of the solutions $\{y_j(t,s;\lambda)\}_{j\in\{0,1\}}$ to the Cauchy problems
\begin{align}\label{CP yj(t,s;lambda,2/3)} 
\begin{cases}  \partial_t^2 y_j(t,s;\lambda) - \lambda^2 t^{-\frac{4}{3}} y_j(t,s;\lambda)= 0, &  t>s\geqslant 1, \\
y_j(s,s;\lambda)= \delta_{0j}, \\
 \partial_t y_j(s,s;\lambda)= \delta_{1j},
\end{cases}
\end{align} where $\lambda>0$ is a parameter and $\delta_{ij}$ denotes the Kronecker delta. Let us introduce the change of variables $z=z(t;\lambda)\doteq -2\lambda \phi(t)$, where for the sake of brevity we denote simply $\phi(t)\equiv \phi_{\frac{2}{3}}(t)= 3t^{1/3}$. Furthermore, we perform the transformation $y(t,\lambda)=w(z) \, \mathrm{e}^{-\frac{z}{2}}$. A straightforward computation shows that
\begin{align*}
\partial_t y(t,\lambda) & = \left[w'(z)-\tfrac 12 w(z)\right] \mathrm{e}^{-\frac{z}{2}} \tfrac{\partial z}{\partial t}, \\
\partial_t^2 y(t,\lambda) & = \left[w''(z)-w'(z)+\tfrac 14 w(z)\right] \mathrm{e}^{-\frac{z}{2}} \left(\tfrac{\partial z}{\partial t}\right)^2 +\left[w'(z)-\tfrac 12 w(z)\right] \mathrm{e}^{-\frac{z}{2}} \tfrac{\partial^2 z}{\partial t^2}.
\end{align*}
 Consequently, $y$ solves the equation
\begin{align}\label{equation y 2/3}
\frac{\mathrm{d}^2 y}{\mathrm{d} t^2} -\lambda^2 t^{-\frac{4}{3}}y=0
\end{align} if and only if $z$ is a solution of the confluent hypergeometric equation
\begin{align}\label{confluent hypergeo eq}
z w''(z)-(z+2)w'(z)+w(z)=0,
\end{align} where we used $\frac{\partial^2 z}{\partial t^2}= 4\lambda t^{-4/3} (\phi(t))^{-1}$ and $\left(\tfrac{\partial z}{\partial t}\right)^2 = 4\lambda^2 t^{-4/3}$. According to \cite[Equation 13.2.32, p. 324]{NIST10}, a fundamental pair of solutions to \eqref{confluent hypergeo eq} is given by $z^3 M(z;2,4)$ and $z+2$. Here $M(z;a,c)$ denotes  Kummer's function $$M(z;a,c)\doteq \sum_{h=0}^\infty \frac{(a)_h}{(c)_h h!} z^h,$$ where $(b)_h$ denotes the Pochhammer symbol (rising factorial) and is defined by $(b)_h=1$ for $h=0$ and $(b)_h= b(b+1) \cdots (b+h-1)$ for $h\geqslant 1$.

\begin{lemma} For any $z\in \mathbb{R}$ the following identity holds
\begin{align}\label{z^3 M(2,4,z)}
z^3 M(z;2,4) = 6 \big( \mathrm{e}^z(z-2)+z+2\big).
\end{align}
\end{lemma}
\begin{proof}
In order to prove \eqref{z^3 M(2,4,z)} we are going to consider the corresponding Taylor series expansions. Let us denote $f(z)\doteq 6 \big( \mathrm{e}^z(z-2)+z+2\big)$. Since 
\begin{align*}
f'(z) & = 6 \big( \mathrm{e}^z(z-1)+1\big), \qquad
f''(z)  = 6 \,  \mathrm{e}^z z,
\end{align*} then, $f(0)=f'(0)=f''(0)=0$.

 Moreover, one can prove recursively that $f^{(2+h)}(z)= 6 \, \mathrm{e}^z(z+h) $ for any $h\geqslant 0$. Therefore,
 \begin{align*}
 f(z) =  \sum_{h=0}^\infty \tfrac{f^{(h)}(0)}{ h!} z^h =   z^3 \sum_{h=0}^\infty \tfrac{f^{(h+3)}(0)}{ (h+3)!} z^{h} =   z^3 \sum_{h=0}^\infty \tfrac{6(h+1)}{ (h+3)!} z^{h}.
 \end{align*} We remark that 
 \begin{align*}
 \tfrac{(2)_h}{(4)_h h!} = \tfrac{(h+1)!}{(1/6) (h+3)!\,  h!}=\tfrac{6(h+1)}{(h+3)!}
 \end{align*} for any $h\in \mathbb{N}$, because of $(2)_h= (h+1)!$ and $(4)_h=(1/6) (h+3)!$. Hence,
 \begin{align*}
 f(z) =  \ z^3 \sum_{h=0}^\infty \tfrac{6(h+1)}{ (h+3)!} z^{h} = \ z^3 \sum_{h=0}^\infty  \tfrac{(2)_h}{(4)_h h!}z^{h} =z^3 M(z;2,4),
 \end{align*} that is, we proved \eqref{z^3 M(2,4,z)}.
\end{proof}

According to our previous remark, by \eqref{z^3 M(2,4,z)} it follows that $6(\mathrm{e}^z(z-2) +z+2)$ and $z+2$ are a fundamental system of solutions for \eqref{confluent hypergeo eq}. For the sake of simplicity we may consider $\{g_1,g_2\}$, where $g_1(z)\doteq \mathrm{e}^z(z-2) $ and $g_2(z)\doteq z+2$ as a basis of the solution space for \eqref{equation y 2/3}. We point out that $\{g_1,g_2\}$ is clearly a fundamental system of solutions as
\begin{align*}
\mathcal{W} (g_1,g_2)(z) = g_1(z)g_2'(z)- g_2(z)g_1'(z) =-z^2 \mathrm{e}^{z}.
\end{align*}  Thus, the pair of functions
\begin{align*}
\widetilde{V}_0(t,\lambda) & \doteq - \tfrac 12 \mathrm{e}^{-\frac{z}{2}}g_1(z) =  \mathrm{e}^{\frac{z}{2}} \left(-\tfrac z2+1\right) = \mathrm{e}^{-\lambda \phi(t)} \left(\lambda\phi(t)+1\right), \\
\widetilde{V}_1(t,\lambda) & \doteq - \tfrac 12 \mathrm{e}^{-\frac{z}{2}}g_2(z) =  \mathrm{e}^{-\frac{z}{2}} \left(-\tfrac z2-1\right) = \mathrm{e}^{\lambda \phi(t)} \left(\lambda\phi(t)-1\right)
\end{align*} form a system of fundamental solutions to \eqref{equation y 2/3}.

Finally, we prove the representations \eqref{def y0(t,s;lambda,2/3)} and \eqref{def y1(t,s;lambda,2/3)} by using $\big\{\widetilde{V}_0,\widetilde{V}_1 \big\}$ as fundamental system of solutions to \eqref{equation y 2/3}.

\begin{proposition} Let $y_0(t,s;\lambda,2/3)$ and $y_1(t,s;\lambda,2/3)$ be the functions defined in  \eqref{def y0(t,s;lambda,2/3)} and \eqref{def y1(t,s;lambda,2/3)}, respectively. Then, $y_0(t,s;\lambda,2/3)$ and $y_1(t,s;\lambda,2/3)$ solve the Cauchy problem \eqref{CP yj(t,s;lambda,2/3)} for $j=0$ and $j=1$, respectively.
\end{proposition}

\begin{proof}
We know that $\widetilde{V}_0,\widetilde{V}_1$ form a system of independent solutions to \eqref{equation y 2/3}. Also, we can write the solutions $y_j(t,s;\lambda)$, $j=0,1$ of \eqref{CP yj(t,s;lambda,2/3)} as linear combinations of $\widetilde{V}_0,\widetilde{V}_1$ in the  following way
\begin{align} \label{representation yj with aj and bj 2/3}
y_j(t,s;\lambda) = a_j(s;\lambda) \widetilde{V}_0(t;\lambda)+  b_j(s;\lambda) \widetilde{V}_1(t;\lambda)
\end{align} for suitable coefficients $a_j(s;\lambda)$ and $ b_j(s;\lambda)$, $j=0,1$.

 The application of the initial conditions $\partial^i_t y_j(s,s;\lambda)=\delta_{ij}$ yields  the system
\begin{align*}
\left(\begin{array}{cc}
\widetilde{V}_0(s;\lambda) & \widetilde{V}_1(s;\lambda)  \\ 
\partial_t \widetilde{V}_0(s;\lambda) & \partial_t \widetilde{V}_1(s;\lambda) 
\end{array} \right) \left(\begin{array}{cc}
a_0(s;\lambda) & a_1(s;\lambda)  \\ 
b_0(s;\lambda) & b_1(s;\lambda) 
\end{array} \right) = I,
\end{align*} where $I$ denotes the identity matrix. Therefore,
\begin{align}
\left(\begin{array}{cc}
a_0(s;\lambda) & a_1(s;\lambda)  \\ 
b_0(s;\lambda) & b_1(s;\lambda)  \end{array} \right) &= \left(\begin{array}{cc}
\widetilde{V}_0(s;\lambda) & \widetilde{V}_1(s;\lambda)  \\ 
\partial_t \widetilde{V}_0(s;\lambda) & \partial_t \widetilde{V}_1(s;\lambda) 
\end{array} \right)^{-1}  \notag \\
 & = \left(\mathcal{W}(\widetilde{V}_0,\widetilde{V}_1)(s;\lambda)\right)^{-1}\left(\begin{array}{cc}
 \partial_t \widetilde{V}_1(s;\lambda)  & -\widetilde{V}_1(s;\lambda)  \\ 
-\partial_t \widetilde{V}_0(s;\lambda) & \widetilde{V}_0(s;\lambda) 
\end{array} \right). \label{inverse matrix 2/3}
\end{align} The Wronskian $\mathcal{W}(\widetilde{V}_0,\widetilde{V}_1)$ is given by
\begin{align*}
\mathcal{W}(\widetilde{V}_0,\widetilde{V}_1)(t;\lambda) & = \widetilde{V}_0(t;\lambda) \partial_t \widetilde{V}_1(t;\lambda)   - \widetilde{V}_1(t;\lambda) \partial_t \widetilde{V}_0(t;\lambda)  = 2\lambda^3 (\phi(t))^2\phi'(t)   = 18 \lambda^3 ,
\end{align*} where  we employed
\begin{align*}
\partial_t \widetilde{V}_0(t;\lambda)  & = -\lambda^2   \phi(t)\, \phi'(t) \,  \mathrm{e}^{-\lambda\phi(t)}, \\
\partial_t \widetilde{V}_1(t;\lambda)  & = \lambda^2   \phi(t)\, \phi'(t) \, \mathrm{e}^{\lambda\phi(t)}.
\end{align*} Plugging the previous representation of  $\mathcal{W}(\widetilde{V}_0,\widetilde{V}_1)$ in \eqref{inverse matrix 2/3}, we find
\begin{align*}
\left(\begin{array}{cc}
a_0(s;\lambda) & a_1(s;\lambda)  \\ 
b_0(s;\lambda) & b_1(s;\lambda)  \end{array} \right) & = \frac{1}{18 \lambda^3}\left(\begin{array}{cc}
 \partial_t \widetilde{V}_1(s;\lambda)  & -\widetilde{V}_1(s;\lambda)  \\ 
-\partial_t \widetilde{V}_0(s;\lambda) & \widetilde{V}_0(s;\lambda) 
\end{array} \right).
\end{align*}
Let us begin by proving that $y_0(t,s;\lambda)=y_0(t,s;\lambda,2/3)$. Employing the above representation of the coefficients $a_0(s;\lambda),b_0(s;\lambda)$ in \eqref{representation yj with aj and bj 2/3}, we have
\begin{align*}
y_0(t,s;\lambda) &= \left(18\lambda^3\right)^{-1} \big\{\partial_t\widetilde{V}_1(s;\lambda) \widetilde{V}_0(t;\lambda)-\partial_t\widetilde{V}_0(s;\lambda) \widetilde{V}_1(t;\lambda)\big\} \\
& = \left(18\lambda^3\right)^{-1} \lambda^2   \phi(s)\, \phi'(s)  \big\{ \mathrm{e}^{-\lambda (\phi(t)-\phi(s))} \left(\lambda\phi(t)+1\right) + \mathrm{e}^{\lambda (\phi(t)-\phi(s))} \left(\lambda\phi(t)-1\right)\big\} \\
& = 3^{-2}   \phi(s)\, \phi'(s) \phi(t) \cosh \big(\lambda (\phi(t)-\phi(s))\big) - 3^{-2}\lambda^{-1}    \phi(s)\, \phi'(s)   \sinh\big( \lambda (\phi(t)-\phi(s))\big) \\
&= (t/s)^{1/3} \cosh \big(\lambda (\phi(t)-\phi(s))\big) - 1/(3\lambda s^{1/3})    \sinh\big( \lambda (\phi(t)-\phi(s))\big) = y_0(t,s;\lambda,2/3). 
\end{align*}
Analogously, plugging the previously determined expressions for $a_1(s;\lambda),b_1(s;\lambda)$ in \eqref{representation yj with aj and bj 2/3}, we have
\begin{align*}
y_1(t,s;\lambda) &= \left(18\lambda^3\right)^{-1} \big\{ \widetilde{V}_0(s;\lambda) \widetilde{V}_1(t;\lambda)-  \widetilde{V}_1(s;\lambda) \widetilde{V}_0(t;\lambda)\big\}  \\
&= \left(18\lambda^3\right)^{-1} \big\{ (\lambda\phi(s)+1)(\lambda\phi(t)-1) \mathrm{e}^{\lambda (\phi(t)-\phi(s))} -  (\lambda\phi(s)-1)(\lambda\phi(t)+1) \mathrm{e}^{-\lambda (\phi(t)-\phi(s))} \big\}  \\
&= \left(18\lambda^3\right)^{-1} (\lambda^2\phi(t)\phi(s)-1) \big( \mathrm{e}^{\lambda (\phi(t)-\phi(s))}- \mathrm{e}^{-\lambda (\phi(t)-\phi(s))}\big)  \\ & \qquad + \left(18\lambda^3\right)^{-1} \lambda(\phi(t)-\phi(s)) \big( \mathrm{e}^{\lambda (\phi(t)-\phi(s))}+ \mathrm{e}^{-\lambda (\phi(t)-\phi(s))}\big) \\
&= \left((st)^{1/3}/\lambda-1/(9\lambda^3)\right) \sinh\big( \lambda (\phi(t)-\phi(s))\big)+ \left(1/9\lambda^2\right) (\phi(t)-\phi(s)) \cosh\big(\lambda (\phi(t)-\phi(s))\big) \\ &=y_1(t,s;\lambda,2/3).
\end{align*}  The proof is complete.
\end{proof}

\end{document}